\documentclass{amsart}
\usepackage{hyperref}

\usepackage{amsmath,color,graphicx,amssymb,amscd}
\usepackage{color}

\makeatletter
\newtheorem{theorem}{Theorem}[section]
\newtheorem{proposition}{Proposition}[section]
\let\c@proposition\c@theorem
\newtheorem{corollary}[theorem]{Corollary}
\newtheorem{lemma}[theorem]{Lemma}
\numberwithin{equation}{section}
\theoremstyle{definition}
\newtheorem{definition}[theorem]{Definition}

\newtheorem{warning}[theorem]{Warning}
\newtheorem{remark}[theorem]{Remark}
\newtheorem{observation}[theorem]{Observation}
\makeatother

\def\L{\mathbf{L}}
\def\K{\mathcal{K}}
\def\M{\mathcal{M}}
\def\N{\mathcal{N}}

\def\A{\mathcal{A}}
\def\Z{\mathbb{Z}}
\def\Q{\mathbb{Q}}
\def\disc{\operatorname{disc}}

\def\Aut{\operatorname{Aut}}
\def\id{\operatorname{id}}
\def\ord{\operatorname{ord}}

\def\iq.{\emph{i.q\PERIOD}}
\def\cf.{\emph{cf.}}
\def\viz.{\emph{viz\PERIOD}}
\def\vs.{\emph{vs\PERIOD}}

\begin{document}

\title[Real Representatives of Equisingular Strata of Simple Quartics]{Real Representatives of Equisingular Strata of Simple Quartic Surfaces}

\author[\c{C}.~G\"{u}ne\c{s} ~Akta\c{s} ]{\c{C}\.{i}sem G\"{u}ne\c{s} Akta\c{s}}
\address{
 Department of Mathematics, Bilkent University\\
 06800 Ankara, Turkey} \email{cisem@fen.bilkent.edu.tr}

\thanks{The author was supported by the \textit{T\"{U}B\.{I}TAK} grant $116$F$211$}

\subjclass[2010]{Primary 14J28; Secondary  14J10, 14J17}

\keywords{Projective model, $K3$-surface, complex quartic, singular quartic}

\date{}

\dedicatory{}

\begin{abstract}

We develop an algorithm detecting real representatives in equisingular strata of projective models of $K3$-surfaces. We apply this algorithm to spatial quartics and find two new examples of real strata without real representatives.
As a by-product, we also give a new proof for the only previously known example of plane sextics.
\end{abstract}

\maketitle


\section{Introduction}\label{introduction}
Throughout the paper, all varieties are over the field $\mathbb{C}$ of complex numbers.

\subsection{Motivation and historical remarks}
It is a wide open problem what kind of singularities a projective surface or a curve of a given degree can have. In general, this problem seems hopeless. However, in the case of $K3$-surfaces, thanks to the global Torelli theorem~\cite{K3} and subjectivity of the period map~\cite{periodmap}, the equisingular deformation classification of surfaces with any given polarization becomes a mere computation. Various deformation classification problems for $K3$-surfaces have been intensively studied in the literature. Any model of $K3$-surfaces is necessarily \emph{simple}, \emph{i.e.}, has at worst simple singularities. The most popular projective models of $K3$-surfaces are sextic curves $C\subset\mathbb{P}^2$ (where a $K3$-surface appears as the double plane ramified along $C$) and quartic surfaces $X\subset\mathbb{P}^3$. A complete list of all possible combinations of simple singularities realized by sextics and quartics was found by Yang~\cite{Yang.sextics,Yang}.  Using the arithmetical reduction~\cite{Alex1}, Shimada~\cite{Shimada.Maximizing} gave a complete description of the moduli spaces of \emph{maximizing} sextics. Later, based on the same approach,  Degtyarev and Akyol~\cite{Alex2} completed the equisingular deformation classification of plane sextics. This approach was also used by G\"{u}ne\c{s} Akta\c{s}~\cite{Cisem1} to obtain a complete description of the equisingular strata of the so-called \emph{nonspecial} simple quartics (see section~\ref{simple.quartics} for the definition). In the meanwhile, Shimada~\cite{Shimada.connEllK3} has listed the connected components of the moduli space of the Jacobian elliptic $K3$-surfaces (which can be regarded as $K3$-surfaces with a $\mathbf{U}$-polarization).

Also worth mentioning is the vast literature on the deformation classification problems in the \emph{real} case, see,\emph{ e.g.}, the classification of real (nonsingular) quartics by Kharlamov~\cite{Kharlamov}, the study on moduli space of real $K3$-surfaces  by Nikulin~\cite{Niku4} or the recent
work on quartic spectrahedra by Degtyarev and Itenberg~\cite{AI} and Ottem \emph{et al.}~\cite{Sturmfels}.

Typically, over $\mathbb{C}$, one deals with singular models, whereas, over $\mathbb{R}$, with very few exceptions (see, \emph{e.g.}~\cite{AI} and~\cite{Sturmfels}), one usually confines oneself to the smooth ones. Certainly, one could have considered singular real models up to equisingular  equivariant deformations, but the results in lists would be huge. (More importantly, one would have to check, on a case by case basis, the validity of the lattice theoretical reduction, which does not hold automatically, see, \emph{e.g.}~\cite{Degtyarev:finiteness}.) In this paper, we make an attempt to bridge this gap from a slightly different perspective, namely, we discuss the existence of real representatives in the real equisingular strata of complex singular models.
\subsection{Principal results}
This paper originates from my paper~\cite{Cisem1}, where I started a systematic equisingular deformation classification of simple quartics. In this paper, based on a more general perspective, we study all models of $K3$-surfaces of a certain fixed \emph{kind} (see 
\autoref{projective.models.of.K3surfaces}).
Denote by $\M$ the space of all models $f\colon X\rightarrow \mathbb{P}^n$; it is divided into equisingular strata $\M(S)$ according to sets $S$ of simple singularities.  Each stratum $\M(S)$  splits further into its connected components, which are the equisingular deformation classes. We will mainly work with the subspace $\M_1\subset\M$ consisting of the nonspecial models and the respective strata $\M_1(S)=\M(S)\cap\M_1$; however, having further applications in mind, we will also discuss the general case whenever possible.

Fix a \emph{real structure} (\emph{i.e.}, an antiholomorphic involution) $\operatorname{conj}\colon \mathbb{P}^n\rightarrow\mathbb{P}^n$, then sending $f\colon X\rightarrow \mathbb{P}^n$ to $\operatorname{conj}\circ f\colon \bar{X}\rightarrow \mathbb{P}^n$ induces a real structure $\operatorname{c}: \M\rightarrow\M$. This real structure $\operatorname{c}$ depends on the choice of $\operatorname{conj}$; however the induced action on the connected components of equisingular strata is well defined. A connected component $\mathcal{D}\subset\mathcal{M}(S)$ is called \textit{real} if $\operatorname{c}(\mathcal{D})=\mathcal{D}$. Clearly, each stratum $\mathcal{M}(S)$ consists of real and pairs of complex conjugate components; this classification of components is given in~\cite{Alex2} for sextics  and in~\cite{Cisem1} for (nonspecial) quartics.

Although it is quite common that a real variety may have no real points, very few examples of equsingular deformation classes with this property are known. Clearly, any class $\mathcal{D}\subset\M(S)$ containing a real model is real. However, the converse is not true, but the only known counterexample is the stratum $\M_1(\mathbf{A}_7\oplus\mathbf{A}_6\oplus\mathbf{A}_5)$ of the space of sextics found in \cite{Alex2}.
In the present paper, we study phenomena of this kind in the space of simple quartics; in particular, we find two more examples as above. Our principal result is the following theorem.

\begin{theorem}\label{principal.result}
Let $\M$ be the space of spatial quartics, and let $S_1=\mathbf{A}_7\oplus\mathbf{A}_6\oplus\mathbf{A}_3\oplus\mathbf{A}_2$ and $S_2=\mathbf{D}_7\oplus\mathbf{A}_6\oplus\mathbf{A}_3\oplus\mathbf{A}_2$. Any real component of any stratum $\mathcal{M}_1(S)$ other than $\mathcal{M}_1(S_1)$ or $\mathcal{M}_1(S_2)$ contains a real surface. The strata $\mathcal{M}_1(S_1)$ and $\mathcal{M}_1(S_2)$ consist of one real component each but they contain no real surfaces.
\end{theorem}
Theorem~\ref{principal.result} is proved in section~\ref{Applications}. As an important by-product of our approach, we give a simpler proof for the following example of Degtyarev and Akyol~\cite{Alex2}.
\begin{proposition}[Proposition 2.6 in~\cite{Alex2}]\label{byproduct}
The stratum $\mathcal{M}_1(\mathbf{A}_7\oplus\mathbf{A}_6\oplus\mathbf{A}_5)$ in the space $\M$ of plane sextics contains no real curves.
\end{proposition}

\subsection{Contents of the paper}
In $\S2$, we recall a few facts of Nikulin's theory of discriminant forms which is the principal technical tool of the paper. In $\S3$, we discuss projective models of $K3$-surfaces, introduce the abstract homological types, and recall the arithmetical reduction of the classification problem (see Theorem~\ref{def.class}). In $\S4$, we restate the existence of real models in arithmetical terms (see Theorem~\ref{real.model}) and suggest two approaches  to find such models: via perturbations and via reflections. In particular, in Proposition~\ref{rankT=2}, we assert that the reflections suffice to detect all real representatives in the submaximal case $\operatorname{rk}\operatorname{NS}(X)=19$. Thus, in $\S5$, we develope a new algorithm listing all involutive skew-automorphism of an abstract homological type inducing a reflection on the transcendental lattice. In $\S6$, this algorithm is applied to two polarizations: spatial quartics (to prove Theorem~\ref{principal.result}) and to plane sextics (to prove Proposition~\ref{byproduct}).
\subsection{Acknowledgements}
I am grateful to Alex Degtyarev for a number of comments, suggestions and fruitful and motivating discussions.


\section{Integral lattices}
\subsection{Finite quadratic forms}
A finite quadractic form is a finite abelian group $\mathcal{L}$ equipped with a map $q\colon \mathcal{L}\rightarrow\mathbb{Q}/2\mathbb{Z}$ such that $q(x+y)=q(x)+q(y)-2b(x,y)$ for all $x,y\in\mathcal{L}$, where $b\colon \mathcal{L}\otimes\mathcal{L}\rightarrow\mathbb{Q}/\mathbb{Z}$ is a symmetric bilinear form (which is determined by $q$) and $2$ is the isomorphism $\times2\colon \mathbb{Q}/\mathbb{Z}\rightarrow\mathbb{Q}/2\mathbb{Z}$. We write $x^2$ and $x\cdot y$ for $q(x)$ and $b(x,y)$, respectively. A finite quadratic form is \emph{nondegenerate} if the homomorphism
\begin{align*}
    \mathcal{L}\rightarrow \operatorname{Hom}(\mathcal{L},\mathbb{Q}/\mathbb{Z}),\quad x\mapsto(y\mapsto x\cdot y)
\end{align*}
is an isomorphism. We denote by $\Aut (\mathcal{L})$ the group of automorphisms of $\mathcal{L}$ preserving the form $q$. A subgroup $\K\subset \mathcal{L}$ is called \emph{isotropic} if the restriction of the quadratic form $q$ on $\mathcal{L}$  to $\K$ is identically zero. If this is case $\K^{\bot}/\K$ also inherits from $L$ a nondegenerate quadratic form.

Each finite quadratic form can be decomposed into the orthogonal direct sum $\mathcal{L}=\bigoplus_p\mathcal{L}_{[p]}$ of its $p$-primary components $\mathcal{L}_{[p]}:=\mathcal{L}\otimes \mathbb{Z}_p$, where the summation runs over all primes $p$. We denote by $\ell(\mathcal{L})$ the minimal number of generators of $\mathcal{L}$ and we put $\ell_p(\mathcal{L})=\ell(\mathcal{L}_{[p]})$. A finite quadratic form $\mathcal{L}$ is called \emph{even} if there is no element $x\in \mathcal{L}_{[2]}$ of order $2$ with $x^2=\pm\frac{1}{2}\bmod 2\Z$.

Given coprime integers $(m,n)$  such that $mn=0\bmod 2$, we denote by $[ \frac{m}{n}]$  the nondegenerate finite quadratic form on $\Z/n\Z$ sending the generator to $\frac{m}{n}\bmod 2\Z $. For a positive integer $k$, we use the notation $\mathcal{U}(2^k)$ and $\mathcal{V}(2^k)$ for the quadratic forms on $\Z/2^k\Z \times \Z/2^k\Z$, defined by the matrices
\begin{align*}
    \mathcal{U}(2^k):=\frac{1}{2^k}\left[
             \begin{array}{cc}
               0 & 1 \\
               1 & 0 \\
             \end{array}
           \right],\quad \mathcal{ V}(2^k):=\frac{1}{2^k}\left[
             \begin{array}{cc}
              2&  1\\
               1 & 2\\
             \end{array}\right].
\end{align*}
(A finite quadratic form can be described by means of the Gram matrix $[\varepsilon_{ij}]$ such that $\varepsilon_{ij}=\beta_i\cdot\beta_j \bmod\Z$ and $\varepsilon_{ii}=\beta_i^2\bmod 2\Z$ where $\beta_k$'s are the basis vectors ). Nikulin \cite{Niku2} proved that, any finite nondegenerate quadratic form decomposes into an orthogonal direct sum of cyclic forms $[\frac{m}{n}]$ and length 2 forms $\mathcal{U}(2^k)$, $\mathcal{V}(2^k)$. We use the notation $\langle\alpha\rangle$ for the the cyclic subgroup generated by $\alpha$.

\begin{definition}\label{detp}
 Let $\mathcal{L}$ be a nondegenerate quadratic form. Given a prime $p$, the determinant of the Gram matrix (in some basis) of $\mathcal{L}_{[p]}$ has the form  $u/|{\mathcal{L}_{[p]}}|$ for some unit $u\in\Z_p^{\times}$, and this unit is independent of the basis modulo $(\Z_p^{\times})^2$ (if $p$ is odd or $\mathcal{L}_{[p]}$ is even) or modulo $(\Z_2^{\times})^2\times\{1,5\}$ (if $p=2$ and $\mathcal{L}_{[2]}$ is odd). We define $\det_p \mathcal{L}=u/|{\mathcal{L}_{[p]}}|$ where $u\in\Z_p^{\times}/(\Z_p^{\times})^2$ or ${u\in\Z_2^{\times}/(\Z_2^{\times})^2\times\{1,5\}}$ is as above (see, \cite{MM3}).
\end{definition}
\begin{remark}\label{NikuDef}
According to Nikulin~\cite{Niku2}, given a prime $p$ and a quadratic form $\mathcal{L}$ on a $p$ group, there is a unique $p$-adic lattice $L$ such that $\operatorname{rk} L= \ell_p(\mathcal{L})$ and $\disc L = \mathcal{L}_{[p]}$. One has $\det L= \det_p\mathcal{L}|{\mathcal{L}_{[p]}}|^2= u |{\mathcal{L}_{[p]}}|$ for  some unit $u$ as in the Definition \ref{detp} (Nikulin uses this equality as a definition of $\det_p \mathcal{L}$).
\end{remark}
\begin{proposition}\label{KpMp}
Let $p$ be a prime and assume that $\mathcal{M}$ is a quadratic form on a $p$-group and  $\K\subset\M$ an isotropic subgroup. If $\ell_p(\K^{\bot}/\K)=\ell_p(\M)$, then $\operatorname{det}_p(\K^{\bot}/\K)=\det_p(\M)\bmod (\mathbb{Q}_p^{\times})^{2}$.
\end{proposition}
\begin{remark}\label{detcongruence}
The equality in Proposition \ref{KpMp} holds in the groups where both determinants are well-defined, \emph{i.e.}, typically in $\mathbb{Q}_p^{\times}/(\mathbb{Q}_p^{\times})^2$; however if $p=2$ and at least one of the forms is odd then the equality holds in $\mathbb{Q}_2^{\times}/(\mathbb{Q}_2^{\times})^2\times\{1,5\}$ (\emph{cf}., Definition \ref{detp}). More precisely, $\operatorname{det}_p(\K^{\bot}/\K)=|{\K}|^2\det_p(\M)\bmod (\mathbb{Z}_p^{\times})^{2}$; in fact, we are comparing the ``essential parts", \emph{i.e.}, units $u$ as in Definition \ref{detp}.
\end{remark}
\begin{proof}
Let $M$ be the $p$-adic lattice as in Remark~\ref{NikuDef} and consider the finite index extension $M'\supset M$ given by the kernel $\K$ (see Proposition~\ref{L-K}). Since $\operatorname{rk}M=\operatorname{rk}M'$ and $\ell_p(\M')=\ell_p(\M)$, the extension  $M'$ can be used to compute $\det \K^{\bot}/\K$. Since $\operatorname{det}(M)=\operatorname{det}(M')|{\K}|^2$, we have the statement.
\end{proof}
The following proposition holds for $p=2$ only.
\begin{proposition}\label{K2M2}
Let $\M$ be a finite quadratic form on a $2$-group and $\K\subset\M$ a cyclic isotropic subgroup. Assume that $\ell_2(\K^{\bot}/\K)<\ell_2(\M)$, then $\ell_2(\K^{\bot}/\K)=\ell_2(\M)-2$ and $\operatorname{det}_2(\K^{\bot}/\K)=-\det_2{\M} \bmod(\mathbb{Q}_2^{\times})^2$ (cf. Remark \ref{detcongruence}).
\end{proposition}
We preceed the proof of Proposition \ref{K2M2} with the following two Lemmas.
\begin{lemma}\label{sequences}
Let $\M$ be a finite quadratic form on a $2$-group and $\mathcal{C}\subset\M$ a cyclic subgroup. Then there exist sequences of integers
\begin{equation*}
    0<m_1<m_2<\cdots<m_N=\log_2|{\mathcal{C}}|
\end{equation*}
and
\begin{equation*}
     0\leq r_1<r_2<\cdots<r_N
\end{equation*}
such that $\M\cong\mathcal{ N}_0\oplus \bigoplus \mathcal{N}_s$, where  $\mathcal{N}_s$ is a nondegenerate finite quadratic form generated by either one element $u_s$ or two elements $u_s, v_s$, all of order $2^{m_s+r_s}$, and whose Gram matrix is
\begin{align*}
  \frac{1}{2^{m_s+r_s}}\left[
       \begin{array}{c}
         \mu_s\\
       \end{array}
     \right]
  \mbox{ or }\frac{1}{2^{m_s+r_s}}\left[ \begin{array}{cc}
                          \mu_s & 1 \\
                           1 &  \nu_s \\
                           \end{array}
                    \right]
\end{align*}
respectively, where $\mu_s$ is odd in the former case  and even in the latter case.
Furthermore, the cyclic subgroup $\mathcal{C}$ is generated by $\kappa=\bigoplus\kappa_s$, where $\kappa_s=2^{r_s}u_s$.
\end{lemma}
\begin{proof}
Using the \emph{partial normal form} (see Lemma 4.2 in \cite{MM3}), we can decompose the quadratic form $\M$ into orthogonal sum $\M=\bigoplus \M_i$, where $\M_i$ is a homogenous group of exponent $2^i$. We construct the sequences $\{m_s\}$, $\{r_s\}$ and $\{\N_s\ni\kappa_s\}$ inductively, starting with $\bar{\kappa}_1:=\kappa$, where $\kappa$ is a generator of the cyclic group $\mathcal{C}$. At step $s$, let $\bar{\kappa}_s=\bigoplus_i\kappa_i'$, $\kappa_i'\in\M_i$, be the corresponding decomposition.  Take
\begin{equation}\label{r1}
    r_s=\operatorname{max}\{r: \bar{\kappa}_s=2^r\alpha,\mbox{ $\alpha\in\M$}\}
\end{equation}
and let
\begin{equation}\label{n.max}
    n=\operatorname{max}\{i:\operatorname{ord}(\kappa'_i)=2^{i-r_s} \}\mbox{ and } m_s=n-r_s=\operatorname{log}_2\operatorname{ord}(\kappa_n').
\end{equation}
Choose $\kappa_s$ as
\begin{equation}\label{kappa1}
    \kappa_s=\bigoplus_{i=\operatorname{ord}(\kappa'_i)\leq {2^{m_s}}}\kappa'_i
\end{equation}
We have $\kappa_s=2^{r_s}u_s$ for some $u_s\in \M$ with $\operatorname{ord}(u_s)=2^n$ and $u_s^2=\lambda/2^{n}$. If $\lambda$ is odd then we take for   $\N_s$ the cyclic group generated by $u_1$ which is an orthogonal summand. If $\lambda$ is even, since $\M_n$ is non-degenerate, there exists $v_s \in\M_n$ such that $u_s\cdot v_s=\frac{1}{2^{n}}$ and we take for $\N_s$ the group generated by $u_s,v_s$.

Now, consider $\bar{\kappa}_{s+1}=\bar{\kappa}_s-\kappa_s$. If $\bar{\kappa}_{s+1}\neq0$, pass to the next step. Eventually, we  obtain sequences $\{m_s\}$, $\{r_s\}$ and $\{\N_s\ni\kappa_s\}$ as in the statement and there remains to let $\N_0=(\bigoplus\N_s)^{\bot}$.


\end{proof}
\begin{lemma}\label{r1=0}
In the notation of Lemma~\ref{sequences}, if $\mathcal{C}$ is isotropic and $\ell_2(\mathcal{C}^{\bot}/\mathcal{C})<\ell_2(M)$, then $r_1=0$.
\end{lemma}
\begin{proof}
Clearly, if $r_s>0$ for some $s$, then  all elements of order $2$ in $\N_s$ are in $\mathcal{C}^{\bot}$. Hence, if $r_1>0$, we have $\ell_2(\mathcal{C}^{\bot})=\ell_2(\M)$ and $\ell_2(\mathcal{C}^{\bot}/\mathcal{C})\geq\ell_2(\M)-1$. Then the assumption  $\ell_2(\mathcal{C}^{\bot}/\mathcal{C})<\ell_2(\M)$ implies that $\ell_2(\mathcal{C}^{\bot}/\mathcal{C})=\ell_2(\M)-1$, which contradicts to the congruence $\ell_2(\mathcal{C}^{\bot}/\mathcal{C})=\ell_2(\M)\bmod 2$.
\end{proof}
\begin{proof}[Proof of Proposition \ref{K2M2}]
We apply Lemma~\ref{sequences} and Lemma~\ref{r1=0} to $\mathcal{C}=\K$. Consider the sequences $r_s, m_s$ and the decomposition $\M\cong\mathcal{ N}_0\oplus \bigoplus \mathcal{N}_s$ given by Lemma \ref{sequences}, we have $r_1=0$ by Lemma \ref{r1=0}. Our goal is to reduce this decomposition to its shortest form. Assuming $N>1$, define $\tilde{\K}:=m_1\K$ and consider $\M'=\tilde{\K}^{\bot}/\tilde{\K}$ and $\K'=\K/\tilde{\K}$. By Lemma \ref{r1=0}, we have $\ell_2(\M')=\ell_2(\M)$. Hence by Proposition \ref{KpMp}, we get $\operatorname{det}_2(\M')=\det_2(\M)$. (Strictly speaking, one should adjust the proof of Lemma \ref{sequences} to show that $\M'$ is even whenever $\M$ is and, hence we do not loose any information, \emph{cf.} Remark \ref{detcongruence}). On the other hand, $\K'^{\bot}/\K'=\K^{\bot}/\K$ and we can replace the pair $(\M,\K)$ by the pair $(\M', \K')$ and apply Lemma \ref{sequences} again. Note that $|{\K'}|=2^{m_1}<|{\K}|$, hence the process bound to converge and we end up with a single essential term decomposition, \emph{ i.e.}, $\M\cong\N_0\oplus\N_1$. Since still $r_1=0$ (by Lemma \ref{r1=0} again), the essential term $\N_1$ is the group generated by $u=\kappa$ and $v$, with the quadratic form given by the Gram matrix
\begin{equation*}
   \frac{1}{2^m} \left[
      \begin{array}{cc}
        0 & 1 \\
       1 & \tau \\
      \end{array}
    \right].
\end{equation*}
Then, clearly, $\K^{\bot}/\K=\N_0$, and it is obvious that $\ell_2(\K^{\bot}/\K)=\ell_2(\M)-2$ and $\operatorname{det}_2(\K^{\bot}/\K)=-\det_2{\M}$.
\end{proof}

\subsection{Integral lattices and discriminant forms}
An \emph{(integral) lattice} is a finitely generated free abelian group $L$ equipped with a symmetric bilinear form $b\colon L\otimes L\rightarrow \mathbb{Z}$. Whenever the form is fixed, we use the abbreviation $x^2=b(x,x)$ and $x\cdot y:=b(x,y)$. A lattice $L$ is called \emph{even} if $x^2:=0\mod 2$ for all $x\in L$; it is called \emph{odd} otherwise. The \emph{determinant} $\det L \in \Z$ is the determinant of the Gram matrix of $b$ in any basis of $L$. Since the transition matrix between any two integral bases has determinant $\pm 1$, the determinant $\det L \in \Z$ is well-defined. A lattice $L$ is called \emph{unimodular} if $\det L=\pm 1$; it is called \emph{nondegenerate} if $\det L \neq 0$, or equivalently, the \emph{kernel}
\begin{align*}
    \operatorname{ker}L=L^{\bot} :=\{x\in L \mid\text{ $x\cdot y= 0$ for all $y\in L$}\}
\end{align*}
is trivial.

Given a lattice $L$, the bilinear form on $L$ can be extended by linearity to a $\mathbb{Q}$-valued bilinear form on $L\otimes\mathbb{Q}$. The inertia indices $\sigma_{\pm}$ of $L$ are the classical inertia indices of $L\otimes\mathbb{Q}$ and  the signature $\sigma L$ is the pair $\sigma L=(\sigma_+L,\sigma_-L)$

 If $L$ is nondegenerate, then the dual group $L^{\vee}:=\operatorname {Hom}(L,\Z)$ can be identified with the subgroup
\begin{align*}
    \{x \in L\otimes\Q \mid \text{$x\cdot y \in \Z$ for all $y \in L$}\}
\end{align*}
There is an obvious canonical inclusion $L=L\otimes \Z\subset L^{\vee}$ and the finite quotient group $\disc L :=L^{\vee}/L$ is called the \emph{discriminant} group. The order of $\disc L$ is equal to $\mathopen|{\det L}\mathclose|$. In particular, $L$ is unimodular if and only if $\disc L=0$.

The discriminant group inherits from $L\otimes\mathbb{Q}$ a nondegenerate symmetric bilinear form
\begin{align*}
    b\colon \disc L\otimes \disc L\rightarrow \mathbb{Q}/\mathbb{Z},\quad (x\bmod \Z)\otimes(y\bmod \Z)\mapsto(x\cdot y)\bmod \Z
\end{align*}
called the \emph{discriminant bilinear form}, and, if $L$ is even, its quadratic extension
\begin{align*}
    q\colon \disc L\rightarrow \mathbb{Q}/2\Z, \quad (x \bmod L) \mapsto x^2 \bmod2\Z,
\end{align*}
called the \emph{discriminant quadratic form}. Note that the discriminant group of an even lattice is a finite quadratic form. We use the notation $\disc_p L$ for the $p$-primary part of $\disc L$.

According to Nikulin \cite{Niku2}, two nondegenerate even lattices $L',L''$ are in the same \emph{genus} if and only if $\sigma L'=\sigma L''$ and $\disc L'\cong\disc L''$ (Here, we skip the original definition of a genus, instead,  we use this criterion). We denote by $g(L)$ the set of all isomorphism classes of nondegenerate even lattices in the genus of $L$. This set is finite (see \cite{Milnor.Sym.bilinear}, the result is due to Milnor but the proof is given in \cite{Serre.Cours}).

The group of autoisometries of a nondegenerate lattice $L$ is denoted by $O(L)$. The action of $O(L)$ extends to $L\otimes\mathbb{Q}$ by linearity, restricts to the dual $L^{\vee}$ and factors to $\disc L$. Therefore, there is a natural homomorphism $O(L)\rightarrow \Aut(\disc L)$. If this does not lead to a confusion, we use the same notation for an autoisometry of $L$ and the induced autoisometry of $\disc L$.

The orthogonal projection of any maximal positive definite subspace in $L\otimes\mathbb{R}$ to any other such subspace is an isomorphism of vector spaces. Hence, all maximal positive definite subspaces in $L\otimes\mathbb{R}$ can be oriented in a coherent way.
A choice of such coherent orientations is called a \textit{positive sign structure} on $L$. We denote by $O^+(L)$ (as opposed to $SO(L)$) the subgroup of $O(L)$ consisting of the isometries preserving a positive sign structure. Either one has $O^{+}(L)=O(L)$ or $O(L)^+$ is a subgroup of $O(L)$ of index $2$. In the latter case, each element of $O(L)\smallsetminus O^+(L)$ is called a \emph{skew-autoisometry} of L,\emph{ i.e.}, skew-autoisometries of $L$ are the autoisometries of $L$ that reverse the positive sign structure.

An important examples of autoisometries are reflections. For a vector $a\in L$, the reflection
\begin{align}
    t_a: x\mapsto x- \frac{2a(x\cdot a)}{a^2} \label{ta}
\end{align}
is well defined if and only if
\begin{align}
    \frac{2a}{a^2}\in L^{\vee}\label{wdfta}.
\end{align}
Note that $t_a$ is an involutive isometry of $L$. If $a^2=\pm1$ or $a^2=\pm2$, then $t_a$ acts identically on $\disc L$ and extends to any overlattice of $L$. Moreover, $a^2>0$ if and only if $t_a$ reverses the positive sign structure.

The \emph{hyperbolic plane} is the lattice $\mathbf{U}:=\mathbb{Z}u\oplus\mathbb{Z}v$, with $u^2=v^2=0$ and $u\cdot v$=1.
\subsection{Root Lattices}\label{root.lattices}
A \emph{root} in an even lattice is a vector of square $(-2)$. A \emph{root lattice} is a negative definite lattice generated by its roots. Each root lattice admits a unique decomposition into an orthogonal direct sum of irreducible ones which are of type $\textbf{A}_n$, $n\geq1$, $\textbf{D}_n$, $n\geq4$, or $\textbf{E}_n$, $n=6,7,8$. For further details on irreducible root systems see \cite{Bour}.

Given a root lattice $S$, we have $O(S)=R(S)\rtimes\operatorname{Sym}(\Gamma)$, where $R(S)\subset O(S)$ is the group generated by reflections against roots and $\operatorname{Sym}(\Gamma)$ is the group of  symmetries of the Dynkin graph $\Gamma_S:=\Gamma$. Let $\operatorname{Sym}'(\Gamma)$ be the group of symmetries of $\textbf{E}_8$-type components. Then the kernel of the map $d\colon O(S)\rightarrow\Aut (\disc S)$ is $R(S)\rtimes\operatorname{Sym}'(\Gamma)$ and, hence,
$d$ admits a partial section,\emph{ i.e.}, an isomorphism
\begin{align}\label{Sym0}
    \operatorname{Im}d \cong\operatorname{Sym}_0(\Gamma)\subset\operatorname{Sym}(\Gamma)\subset O(S)
\end{align}
where $\operatorname{Sym}_0(\Gamma)$ is the group of symmetries acting identically on the union of $\textbf{E}_8$-type components.
\subsection{Lattice extensions}\label{lattice.extensions}
From now on, unless specified otherwise, all lattices considered are nondegenerate and even.
An \emph{extension} of a lattice $S$ is an over lattice $L\supset S$. An \emph{isomorphism} between two extensions $L',L''$ is a bijective isometry $L'\rightarrow L''$ identical on $S$. More generally, for a given subgroup $G\subset O(S)$, we define $G$-\emph{isomorphisms} of extensions of $S$ as those which restrict to an element of $G$ on $S$.

Given a \emph{finite index extension} $L\supset S$ (\emph{i.e.}, $S$ is a finite index subgroup of $L$), there is a unique embedding $L\subset S\otimes \mathbb{Q}$. Then we have a chain of inclusions
\begin{align*}
    S\subset L\subset L^{\vee}\subset S^{\vee}.
\end{align*}
The subgroup $\K:=L/S\subset S^{\vee}/S=\disc S$ is called the \emph{kernel} of the finite index extension $L\supset S$. Since $L$ is an even integral lattice, the restriction to $\K$ of the quadratic form $q$ on $\disc S$ is trivial, \emph{i.e.}, $\K$ is \emph{isotropic}. Conversely, given an isotropic subgroup $\K\subset \disc S$, the lattice $L:=\{x\in S\otimes \mathbb{Q} \mid x\bmod S\in \K\}$ is an extension of $S$. Hence, we have the following result.
\begin{proposition}[Nikulin \cite{Niku2}]\label{L-K}
Let $S$ be a nondegenerate even lattice, and fix a subgroup $G\subset
O(S)$. The map $L\mapsto \mathcal{K}=L/S \subset \disc S$
establishes a one-to-one correspondence between the set of
$G$-isomorphism classes of finite index extensions $L\supset S$
and the set of $G$-orbits of isotropic subgroups
$\mathcal{K}\subset \disc S$. Under this correspondence one
has $\disc L=\mathcal{K}^{\bot}/\mathcal{K}$. Furthermore an autoisometry of $S$ extends to a finite index extension $L\supset S$ if and only if it preserves $\K$.
\end{proposition}
An extension $L\supset S$ is called \emph{primitive} if $L/S$ is torsion free. Clearly, $L$ is a finite index extension of $S\oplus N$, where $N:=S^{\bot}$ is also primitive in $L$, and by Proposition \ref{L-K}, it is described by its kernel
\begin{equation*}
    \K\subset \disc (S\oplus N)=\disc S\oplus \disc N.
\end{equation*}
Since $S$ and $N$ are both primitive in $L$, the kernel $\mathcal{K}$  does not intersect with any of $\disc S$ and $\disc N$. It follows that the projection maps
\begin{equation*}
    \operatorname{proj}_{S}\colon \K \rightarrow \disc S\mbox{ and } \operatorname{proj}_{N}\colon \K \rightarrow \disc N
\end{equation*}
are both monomorphisms. Since $\K$ is isotropic, it is the graph of a bijective anti-isometry $\psi\colon \mathcal{S}'\rightarrow\mathcal{N}'$, where $\mathcal{S}'= \operatorname{proj}_{S}(\K)$ and $\mathcal{N}'= \operatorname{proj}_{N}(\K)$. Conversely, given a bijective anti-isometry $\psi\colon\mathcal{S}'\rightarrow\mathcal{N}'$ where $\mathcal{S}'\subset\disc S$ and $\mathcal{N}'\subset \disc N$, the graph of $\psi$ is an isotropic subgroup $\mathcal{K}\subset\disc S\oplus \disc N$ and the corresponding finite index extension $L\supset S\oplus N$ is a primitive extension whose kernel is $\K$. Thus, we have the following statement (\cf. Nikulin~\cite{Niku2}).
\begin{lemma}\label{genlemma}
Given two nondegenerate even lattices $S$, $N$ and a subgroup $G\subset
O(S)\times O(N)$, there is a one-to-one correspondence between the set of $G$-isomorphism classes of finite index extensions $L\supset S\oplus N$  in which both $S$ and $N$ are primitive  and that of $G$-conjugacy classes of bijective anti-isometries
\begin{equation}\label{isodisc}
    \psi\colon \mathcal{S}'\rightarrow \mathcal{N}'
\end{equation}
where $\mathcal{S}'\subset \disc S$ and $\mathcal{N}'\subset \disc N$. Furthermore, a pair of isometries $f_1\in O(S)$ and $f_2\in O(N)$ extends to $L$ if and only if $f_1|_{\mathcal{S}'}=\psi^{-1}f_2|_{\mathcal{N}'}\psi$ in $\Aut (\mathcal{S}')$.
\end{lemma}
If $L$ above is unimodular, $\disc L=0$, we have
$\mathopen|\disc S\mathclose|\mathopen|\disc N\mathclose|=\mathopen|\mathcal{S}'\mathclose|\mathopen|\mathcal{N}'\mathclose|$. Hence, $\mathcal{S}'=\disc S$ and $\mathcal{N}'=\disc N$ and $\psi$ in \eqref{isodisc} is an anti-isomorphism $\disc S\rightarrow\disc N$. Since also $\sigma_{\pm}N=\sigma_{\pm}L-\sigma_{\pm}S$, it follows that the genus $g(N)$ is determined by the genera $g(S)$ and $g(L)$; we will denote this common genus by $g(S^{\bot}_L)$ (We emphasize that $g(S^{\bot}_L)$  merely encodes a ``local data" composed formally from $g(S)$ and $g(L)$; apriori, it may even be empty, \emph{cf.} Theorem \ref{th.N.existence} below). If $L$ is also indefinite, it is unique in its genus (see, \emph{e.g.}, Siegel~\cite{SiegelI,SiegelII,SiegelIII}). Then, given  a subgroup $G\subset O(S)$ and unimodular even indefinite lattice $L$, a $G$-isomorphism class of a primitive extension $L\supset S$ is determined by a choice of
\begin{enumerate}
\item an even lattice $N\in g(S^{\bot}_L)$, and
\item a bi-coset in $G\backslash \Aut (\disc N)/O(N)$.
\end{enumerate}
In particular the extension $L\supset S$ exists if and only if the genus $g(S^{\bot}_L)$ is nonempty.

From now on we fix the notation $\textbf{L}:= 2\textbf{E}_8\oplus3\textbf{U}$. Note that $2\textbf{E}_8\oplus3\textbf{U}$ is the unique even unimodular lattice of signature $(3,19)$.  We are concerned about this lattice, since it is the intersection index form of a $K3$-surface. More precisely, we are interested in the embeddings to this lattice $\L$. For the ease of the references, we recast a special case of Nikulin's existence theorem as a criterion for $g(S^{\bot}_{\textbf{L}})\neq\emptyset$.
\begin{theorem}[Nikulin~\cite{Niku2}]\label{th.N.existence}
Given a nondegenerate even lattice~$S$, a primitive extension $\textbf{L}\supset S$ exists
if and only if
the following conditions hold
\begin{enumerate}
\item $\sigma_+ S\leq 3$, $\sigma_-S\leq 19$ and $\ell(\mathcal{S})\leq 22- \operatorname{rk} S$, where $\mathcal{S}=\disc S$;
\item one has  $|{\mathcal{S}}|\det_p (\mathcal{S}) =(-1)^{\sigma_+S-1} \bmod (\Z_p^{\times})^2$ for each  odd prime $p$ such that  $\ell_p(\mathcal{S})= 22- \operatorname{rk} S$;
\item If $\ell_2(\mathcal{S})= 22- \operatorname{rk} S$, and $\mathcal{S}_2$ is even  then $|{\mathcal{S}}|\det_2 (\mathcal{S}) =\pm 1 \bmod (\Z_2^{\times})^2$.
\end{enumerate}
\end{theorem}
\section{Projective models of K3-Surfaces}\label{projective.models.of.K3surfaces}
In this section, we consider the \emph{projective models} of a smooth $K3$-surface $X$, \emph{i.e.}, the morphisms $f_h\colon X\rightarrow \mathbb{P}^{n+1}$ defined by a complete linear system $|h|$ without fixed components and such that  $h\in NS(X)\subset H_2(X;\Z)$ and $h^2=2n>0$. Given a projective model $f_h\colon X\rightarrow \mathbb{P}^{n+1}$, the class $h$ is called the \emph{polarization}. Note that $\operatorname{dim}f_h(X)=2$. It can be found in \cite{Donat} that only the following two cases can happen:
\begin{enumerate}
  \item either $f_h$ is birational mapping of $X$ onto a surface of degree $2n$,
  \item or $f_h$ is two-to-one mapping of $X$ onto a surface of degree $n$.
\end{enumerate}
A projective model $f_h$ as in  $(1)$ is called \emph{birational} and it is is called \emph{hyperelliptic} if it is as in $(2)$. By Saint-Donat~\cite{Donat}, we have the following result about hyperelliptic models of K3-surfaces:
\begin{proposition}
The projective model $f_h\colon X\rightarrow \mathbb{P}^{n+1}$ with $h^2=2n$ is hyperelliptic if and only if
\begin{itemize}
  \item [(i)] either $n=1$ or,
  \item [(ii)] $n=4$ and $h=2h'$ for some vector $h'\in NS(X)$ ($h$ is imprimitive) or,
  \item [(iii)] there is a vector $e\in NS(X)$ such that $e^2=0$ and $e\cdot h=2$ .
\end{itemize}
\end{proposition}
Note that the intersection lattice $L_X:=H_2(X;\Z)$ is of the form
\begin{equation*}
    L_X=H_2(X;\Z)\cong \L=2\textbf{E}_8\oplus3\textbf{U}.
\end{equation*}
Given a projective model $f_h\colon X\rightarrow \mathbb{P}^{n+1}$, we introduce the following notations:
\begin{itemize}
  \item $S_X\subset L_X$: the sublattice generated by the curves contracted by  $f_h$;
  \item $S_{X,h}:=S_X\oplus \mathbb{Z}h_X\subset L_X$ where $h_X=h\in NS(X)$ is the class of the pull-back of a generic plane section of $X$;
  \item $\tilde{S}_X \subset \tilde{S}_{X,h}\subset L_X$: the primitive hulls of $S_X$ and $S_{X,h}$, respectively, \textit{i.e}, $\tilde{S}_X:=(S_X\otimes\mathbb{Q})\cap L_X$ and  $\tilde{S}_{X,h}:=(S_{X,h}\otimes\mathbb{Q})\cap L_X$;
  \item $\omega_X \subset L_X\otimes \mathbb{R}$: the oriented $2$-subspace spanned by the real and imaginary parts of the class of a holomorphic $2$-form on $X$ (the \emph{period} of $X$).
\end{itemize}
Recall that all singularities are \emph{simple}
and, hence, $S_X$ is a \emph{root} lattice, \emph{i.e.}, a negative definite lattice generated by vectors of square $(-2)$ (\emph{roots}).
The triple $(S_X,h_X,L_X)$ is called the \emph{homological type} of the projective model $f_h\colon X\rightarrow \mathbb{P}^{n+1}$. This triple has certain properties depending on the ``kind" of the model. Here, fixing the ``kind of the model" includes
\begin{itemize}
  \item [(i)] fixing the degree,
  \item [(ii)]deciding whether the model is birational or hyperelliptic and
  \item [(iii)]sometimes, assuming some additional geometric properties detectable homologically, \emph{e.g.}, presence or absence of certain classes, \emph{cf}. Definition~\ref{badvectors}
\end{itemize}
To capture those  properties of a homological type of certain kinds of models, we have the following definition:
\begin{definition}\label{badvectors}
Let $L$ be a lattice isomorphic to $\L$. Depending on the geometric problem, we define   ``bad vectors'' of a polarized sublattice $S_h\subset L$ containing a distinguished vector $h$ with $h^2=2n$,  as the vectors $e\in \tilde{S}_h:=(S_h\otimes\Q)\cap L$ satisfying one of the following properties:
\begin{enumerate}
  \item $e^2=0$ and $e\cdot h= 1$ (fixed components);
  \item $e^2=0$ and $e\cdot h= 2$ (linear generatricies);
  \item $e^2=0$ and $e\cdot h= 3$ (cubic equations);
  \item $n=4$ and $2e=h$ (Veronese polarization).
\end{enumerate}
\end{definition}
\begin{remark}
 Note that the existence of a vector as in Definition \ref{badvectors}(2) implies the existence of a vector as in (1) in the definition. Hence, usually only the vectors as in (2) are mentioned  whereas the vectors as in (3) are excluded if $n=4$ and we want to consider thriquadrics rather than all octic surfaces.
\end{remark}
The definition below is intended to capture the necessary arithmetical properties of the homological types of models of K3-surfaces. Therefore it depends on the kinds of models considered which is assumed to be fixed in advance.
\begin{definition}\label{abstract.homological.type}
Let $S$ be a root lattice and $n$ be an integer such that $n\geq1$. An \emph{abstract homological type (extending $S$)} \emph{associated} to a given kind of model is an extension of $S_h:=S\oplus\Z h$, $h^2=2n$, to a lattice $L$ isomorphic to $\L$ satisfying the following conditions:
\begin{enumerate}
  \item each vector $e\in(S\otimes\mathbb{Q})\cap L$ with $e^2=-2$ and $e\cdot h=0$ is in $S$;
  \item depending on the kind of the models considered, the primitive hull $\tilde{S}_h:=(S_h\otimes\Q)\cap L$ should not contain the ``\emph{bad vectors}" (specified on the case by case basis, see below). Most notably,
  \begin{itemize}
    \item vectors as in Definition \ref{badvectors}(1) are always excluded,
    \item vectors as in Definition \ref{badvectors}(2) and (4) are excluded if and only if the model is birational.
  \end{itemize}
\end{enumerate}
\end{definition}
Commonly used birational projective models/polarizations are as follows (where we mention also precise type of \emph{bad vectors}, by referring to the names introduced in the Definition \ref{badvectors}, that are to be excluded in Definition \ref{abstract.homological.type}).
\begin{enumerate}
 \item $h^2=4$: The image of $f_h\colon X\rightarrow \mathbb{P}^3$ is a quartic (spatial model); the excluded bad vectors are fixed components as in (1) and linear generatrices as in (2).
 \item $h^2=6$: The image of  $f_h\colon X\rightarrow \mathbb{P}^4$ is a sextic given by a complete intersection of a quadric and cubic (sextic model); the excluded bad vectors are fixed components as in (1) and linear generatrices as in (2).
 \item $h^2=8$: The image of  $f_h\colon X\rightarrow \mathbb{P}^5$ is an octic (octic model); in the most general case the excluded bad vectors are fixed components as in (1), linear generatrices as in (2) and Veronese polarizations as in (4). We can distinguish triquadric vs. all octics; in the former case the bad vectors as in (3) are also to be excluded.
\end{enumerate}
Commonly used    hyperelliptic projective models are as follows:
\begin{enumerate}
   \setcounter{enumi} 3  \item $h^2=2$: The map $f_h\colon X\rightarrow \mathbb{P}^2$ is a degree $2$ map ramified at a sextic curve $C\subset \mathbb{P}^2$ (planar model); the excluded bad vectors are fixed components as in (1).
     \item $h^2=4$: The map $f_h\colon X\rightarrow \mathbb{P}^1\times\mathbb{P}^1$ is a degree $2$ map ramified at a curve $C\subset \mathbb{P}^1\times\mathbb{P}^1$ of bidegree $(4,4)$; the excluded bad vectors are the fixed components as in (1), whereas at least one linear generatrice as in (2) are \textit{assumed}.
\end{enumerate}
In section~\ref{Applications}, we consider examples of the planar model with $h^2=2$ and the spatial model with $h^2=4$.

\marginpar{\em{\color{red}}}
We use the notation $\mathcal{H}=(S\oplus \Z h\subset L)$ for an abstract homological type extending  the root lattice $S$. An abstract homological type $\mathcal{H}=(S\oplus\Z h\subset L)$ is called \emph{maximizing} if $\operatorname{rk}S=19$ (the maximal possible), \emph{i.e.}, $\operatorname{rk}S^{\bot}_h=2$.
An \textit{isomorphism} between two abstract homological types $\mathcal{H}_i=(S_i\oplus \Z h_i\subset L_i)$, $i=1,2$, is an isometry $L_1\rightarrow L_2$, taking $h_1$ to $h_2$ and $S_1$  to $S_2$ (as a set). A \emph{skew-automorphism} of an abstract homological type $\mathcal{H}=(S\oplus\Z h\subset L)$ is a skew-autoisometry of $L$ preserving $S$ (as a set) and $h$.

Given an abstract homological type $\mathcal{H}=(S\oplus \Z h\subset L)$, the group of autoisometries of the primitive hull $\tilde{S}_h=(S_h\otimes\Q)\cap L$  preserving $h$ is denoted by $O_h(\tilde{S}_h)$. Obviously we have
\begin{equation*}
   O_h(\tilde{S}_h)\subset O_h(S_h)=O(S).
\end{equation*}
Note that $S_h^{\bot}$ is a nondegenerate lattice with $\sigma_+S_h^{\bot}=2$, hence a choice of an orientation of one positive definite $2$-subspace in $S_h^{\bot}\otimes\mathbb{R}$ defines a coherent orientation of any other.
\begin{definition}
An \emph{orientation} of an (abstract) homological type $\mathcal{H}=(S\oplus \Z h\subset L)$ is a positive sign structure $\theta$ on $S_h^{\bot}$. Oriented abstract homological types $(\mathcal{H}_i,\theta_i)$, $i=1,2$, are \emph{isomorphic} if there is an isomorphism $ \mathcal{H}_1\rightarrow\mathcal{H}_2$ taking $\theta_1$ to $\theta_2$. The type $\mathcal{H}$ is called \emph{symmetric} if it admits a skew-automorphism, \emph{i.e.}, $(\mathcal{H},\theta)\cong(\mathcal{H},-\theta)$ for some orientation $\theta$ of $\mathcal{H}$.
\end{definition}

Due to Saint-Donat~\cite{Donat} and Urabe~\cite{Urabe2}, a homological type $\mathcal{H}_X=(S_X,h_X,L_X)$ of a projective model $f_h\colon X\rightarrow \mathbb{P}^{n+1}$ is an abstract homological type (as in the Definition \ref{abstract.homological.type}). Then, the period $\omega_X$ of $X$ defines an orientation of $\mathcal{H}_X$.
\begin{theorem}[\emph{cf.} Theorem 2.3.1 in \cite{AI}]\label{def.class}
The map sending a projective model $f_h\colon X\rightarrow\mathbb{P}^{n+1}$ to its oriented homological type establishes a one-to-one correspondence between equisingular deformation classes of models of a certain fixed kind $\mathfrak{K}$ (with a fixed set of simple singularities $S$) and orientation preserving isomorphism classes of oriented abstract homological types (extending $S$) associated to $\mathfrak{K}$. The homological types of complex conjugate strata differ by orientations.
\end{theorem}

\section{Real structures}
\subsection{Real models}
Let X be a complex analytic variety. A \emph{real structure} on X is an anti-holomorphic map $c\colon X\rightarrow X$ which is an involution. The fixed point set $X_{\mathbb{R}}:=\operatorname{Fix} c$ is called the \emph{real part} of $X$. A subvariety $Y\subset X$ is called \emph{real} if $c(Y)=Y$. A \emph{real model} is a pair $(f_h,c)$, where $f_h \colon X\rightarrow \mathbb{P}^n$ is a projective model of a smooth $K3$-surface $X$ and $c$ is a real structure on $X$ preserving $h$, \emph{i.e.}, such that $c(h)=-h$. Note that for a projective model $f_h \colon X\rightarrow \mathbb{P}^n$, we have $\mathbb{P}^n=|h|^{\vee}$, hence the real structure $c$ gives a real structure on $\mathbb{P}^n=|h|^{\vee}$.

Recall that, if $n$ is even, there is one \emph{standard} real structure up to isomorphism (and up to deformation equivalence) on $\mathbb{P}^n$ given by the standard complex conjugation $\operatorname{conj}\colon\mathbb{P}^n\rightarrow\mathbb{P}^n$, $z=(z_0:z_1:\ldots:z_{n})\mapsto \bar{z}=(\bar{z}_0:\bar{z}_1:\ldots:\bar{z}_{n})$ in appropriate homogeneous coordinates. If $n=2k+1$ is odd, there are two real structures  up to isomorphism (and up to deformation equivalence) on $\mathbb{P}^n$.
One of them is the standard one $\operatorname{conj}\colon\mathbb{P}^n\rightarrow\mathbb{P}^n$, $z\mapsto\bar{z}$ mentioned above and the second one is given by $c_2\colon\mathbb{P}^{2k+1}\rightarrow\mathbb{P}^{2k+1}$, $(z_0:z_1:\ldots:z_{2k}:z_{2k+1})\mapsto (\bar{z}_1:-\bar{z}_0:\ldots:\bar{z}_{2k+1}:-\bar{z}_{2k})$. Note that the nonstandard real structure $c_2$ on $\mathbb{P}^{2k+1}$ has empty real part, \emph{i.e.}, $\operatorname{Fix}(c_2)=\emptyset$.

\begin{remark}
Given a real projective model $f_h \colon X\rightarrow \mathbb{P}^n$ where $n$ is odd, it is not obvious which one of the real structures $\operatorname{conj}$, $c_2$ (as above) on $\mathbb{P}^n=|h|^{\vee}$ is induced by the real structure on $X$. This question is answered by Kharlamov \cite{Kharlamov}, in terms of the induced action $c_{*}$ on the polarized lattice $(H_2(X),h)$.
\end{remark}
The following theorem is the arithmetical reduction of the main problem of finding real models.
\begin{theorem}[see Theorem $6.1$ in \cite{Alex2}]\label{real.model}
An abstract oriented homological type $\mathcal{H}$ is realized by a real model if and only if $\mathcal{H}$ admits an involutive skew-automorphism.
\end{theorem}
\subsection{Finding real representatives}
By Theorem \ref{real.model}, to obtain a real model, we will attempt to find involutive skew-automorphisms of the abstract homological type $\mathcal{H}=(S\oplus\Z h\subset L)$.
This problem is straightforward for the maximizing abstract homological types; for the others, we discuss two approaches, via perturbations of abstract homological types and via reflections.

\emph{A (formal) perturbation} of an abstract homological type $\mathcal{H}=(S\oplus\Z h\subset L)$ is any abstract homological type $\mathcal{H}'=(S'\oplus\Z h\subset L)$ such that $S'\subset S$ is primitive in $S$ and the embedding $S'\subset L$ is the restriction of the embedding $S\subset L$. Note that any perturbation of a primitive abstract homological type is also primitive. Most abstract homological types can be obtained by a perturbation from maximizing homological types and this phenomena can be used by means of the following proposition.
\begin{proposition}\label{real.pert}
If an abstract homological type $\mathcal{H}$ is realized by a real birational model $(f_h,c)$, then any $c$-invariant perturbation $\mathcal{H}'$ is also realized by a real model.
\end{proposition}
\begin{remark}
In the computations in chapter~\ref{Applications}, the above proposition holds for birational models; for the hyperelliptic case a stronger statement can be found in \cite{Alex2}.
\end{remark}
The following simple sufficient condition is in fact also necessary for the existence of a real model in which all exceptional divisors are real.
\begin{proposition}
Let $\mathcal{H}=(S\oplus\Z h\subset L)$ be an abstract homological type. If the transcendental lattice $T:=S_h^{\bot}$ contains a sublattice isomorphic to $[2]$ or $\mathcal{U}(2)$ then $\mathcal{H}$ is realized by a real model. Conversely, if $\mathcal{H}$ admits a real model under which all exceptional divisors are real, $T$ contains a sublattice isomorphic to $[2]$ or $\mathcal{U}(2)$.
\end{proposition}
\begin{proof}
If $A=[2]$ or $\mathcal{U}(2)$ is contained in $T=S_h^{\bot}$, then the pair $(-\id, \id)$ on $A\oplus A^{\bot}$ extends to $L$ by Lemma \ref{genlemma}; this extension is an involutive skew-automorphism which implies  the statement by Theorem~\ref{real.model}. Conversely, let
 $c\colon L\rightarrow L$ be an automorphism induced by a real structure as in the statement ($-c$ is an involutive skew-automorphism) and denote by $L_{\pm c}$ its eigenlattices. Note that $\sigma_+L_{-c}=2$ and $h\in L_{-c}$, hence $S\subset L_{-c}$ which implies $T\supset L_{+c}$. Then by Nikulin's classification of real structures on $K3$-surfaces (see \cite{Niku2}), $L_{+c}$ contains $[2]$ or $\mathcal{U}(2)$ and, hence, so does $T$.
\end{proof}
Let $\mathcal{H}=(S\oplus\Z h\subset L)$ be an abstract homological type and $T:=S_h^{\bot}$ be the transcendental lattice. We discuss the existence of real models realizing $\mathcal{H}$ case by case in terms of rank of $T$, explaining the role of reflections (the existence of which is established in the next section).
\subsubsection{The case $\operatorname{rk}T=2$}
The lattice $T$ is a positive definite lattice of rank $2$. By a classical and well known result of Gauss \cite{Gauss}, the group of isometries $O(T)$ is a finite group, which is easily computable. In fact, it turns out that any skew-autoisometry of $T$ is a reflection. (Note, though, that our approach for finding reflections in section \ref{section.reflection} does not apply here since $T$ is often not unique in its genus)
\begin{proposition}\label{rankT=2}
Any symmetric maximizing abstract homological type (extending $S$) admits an involutive skew-automorphism; equivalently,  any real component of the  strata $\mathcal{M}_1(S)$ contains a real model.
\end{proposition}
\begin{proof}
Let $\mathcal{H}=(S\oplus\mathbb{Z} h\subset L)$ be a maximizing symmetric abstract homological type. Since  any skew-autoisometry $r$ of $T=S^{\bot}_h$ is a reflection ($\operatorname{rk}T=2$), it acts as an involution on $\disc T\cong -\disc S$. Then by \eqref{Sym0}, the resulting involution in $\Aut(\disc S)$ is realized by an involution $r'\in \operatorname{Sym}_0(\Gamma_S)\subset O(S)$ (see section \ref{root.lattices}) and $r\oplus r'$ extends to an involution on $L$.
\end{proof}
\subsubsection{The case $\operatorname{rk}T=3$}
Typically, the group $O(T)$ is infinite; however we have the following simple characterization of involutive skew-autoisometries.
\begin{proposition}\label{rankT=3}
Let $T=S_h^{\bot}$ be a lattice of rank $3$. Then any involutive skew-autoisometry $r$ of $T$ is of the form $\pm r'$  where $r'$ is a reflection on $T$ .
\end{proposition}
\begin{proof}
We denote by $T_{\pm r}$ the $(\pm 1)$-eigenspaces of $r$. Recall that the involution $r$ is a reflection if and only if $\operatorname{dim}(T_{-r})=1$. Since here $\operatorname{rk}T=3$, one can have either $\operatorname{dim}(T_{-r})=1$ or $\operatorname{dim}(T_{+r})=1$ or $T_{-r}=0$ (and $r=\id$) or $T_{+r}=0$ (and $r=-\id$). In the last two cases $r$ is not a skew-autoisometry. If $\operatorname{dim}(T_{-r})=1$, the involution $r$ itself is a reflection and if $\operatorname{dim}(T_{+r})=1$, the map $-r$ is a reflection.
\end{proof}
\begin{remark}\label{plusminus.h}
If necessary, we multiply all the maps by $-1$ to make sure that the involutive skew-autoisometry  of $T$ is a reflection. Hence, in this case, we have to extend the group $O_h(S_h)$ to $O_{\pm h}(S_h):=O_h(S_h)\times\{\pm \id_h\}$ allowing the involution $h \mapsto -h$.
\end{remark}
Thus, by Proposition \ref{rankT=3} (and Remark~\ref{plusminus.h}), reflections on $T$ are enough to prove both the existence and non-existence of a real model realizing the strata in this case.
\subsubsection{The case $\operatorname{rk} T\geq 4$}
In this case, we can no longer guarantee that any involutive skew-autoisometry of $T$ is a reflection. However, in all examples considered in the paper, it turns out that each symmetric homological type with $\operatorname{rk}T\geq4$ does admit an involutive skew-autoisometry which is a reflection on $T$; thus, it appears that reflections still suffice to conclude the realizability of real strata by real models.
\section{Real structures via reflections}\label{section.reflection}
\subsection{The set-up}
Fix a primitive sublattice $M\subset L\cong\L$ and let $N:= M^\bot$ be its orthogonal complement. Fix also a subgroup $G\subset O(M)$.  Consider the finite index extension
\begin{equation}\label{L2}
    M\oplus N \subset L,
\end{equation}
both $M$ and $N$ being primitive in $L$. Our aim is to search for
 \begin{align}\label{aim}
    \text{an involution $\varphi\in G$ on $M$ such that $\varphi\oplus t_a$ extends to $L$}
\end{align}
where $a\in N$ is a primitive vector satisfying \eqref{wdfta} such that $a^2>0$ (since we want a map that reverses positive sign structure) and $t_a$ is a reflection as in \eqref{ta}. By \eqref{wdfta}, we have an apriori  bound
\begin{align}\label{adivides2exp}
    a^2\mathrel|2\operatorname{exp}(\disc N)
\end{align}
where $\operatorname{exp}(\disc N)$ is the exponent of the group $\disc N$.

Let $N'$ be the orthogonal complement of the primitive vector $a$,  \emph{i.e.,} $N':= a^{\bot}\subset N$. Then $N$ is a finite index extension of $\Z a\oplus N'$ and we have
\begin{equation}\label{L3}
    M\oplus\Z a\oplus N'\subset L.
\end{equation}
Hence to study finite index extension $L \supset M\oplus N$ as in \eqref{L2}, one can first study the finite index extension
\begin{equation*}
   N\supset\Z a\oplus N'.\label{fieT}
\end{equation*}
However, there is another approach: We start by analyzing finite index extension
\begin{equation}\label{fieL}
    \widetilde{M}_a:=(M\oplus\Z a)\otimes \Q \cap L\supset M\oplus\Z a=:M_a
\end{equation}
Then we have $L\supset \widetilde{M}_a\oplus N'$. Note that
\begin{align}\label{tildeprimitive}
    \text{$M\subset \widetilde{M}_a$ and $\Z a\subset \widetilde{M}_a$ are both primitive}.
\end{align}
Furthermore,
\begin{align}\label{inducesid}
    \text{$\varphi\oplus t_a$ should induce $\id$ on $\disc \widetilde{M}_a$}.
\end{align}
\begin{remark}
Strictly speaking, the approach above gives us the reflections $t_a$ in lattices which are in the genus of $N$. However, in our calculations usually $N$ is unique in its genus (see, section \ref{Applications}).
\end{remark}
Let $\A:=\disc\Z a\cong[\frac{1}{a^2}]$ generated by $\alpha:=a/a^2$. By Proposition \ref{L-K}, there is a one-to-one correspondence between the set of isomorphism classes of finite index extensions $\widetilde{M}_a\supset M_a$ satisfying certain properties and that of isotropic subgroups
\begin{equation*}
    \K \subset \disc (M\oplus \Z a)=\operatorname{disc}M\oplus \A,
\end{equation*}
and we have $\disc \widetilde{M}_a=\K^{\bot}/\K$ (see Lemma~\ref{genlemma}).
 \begin{lemma}\label{main.lemma}
 A pair $\K\subset \disc M \oplus \A$ and $\varphi\in O(M)$, $\varphi^2= \id$ satisfy conditions \eqref{tildeprimitive} and \eqref{inducesid} above if and only if
 \begin{enumerate}
   \item $\K$ is the cyclic group generated by $\vartheta:=\kappa\oplus na/a^2$ where $n=1$ or $2$ and $\kappa\in \disc M$ is such that
   $\operatorname{order}(\kappa)=a^2/n$ and $\kappa^2=-n^2/a^2$,
  \item $\varphi(\kappa)=-\kappa$ in $\disc M$ and,
  \item $\varphi\oplus t_a$ induces $\id$ on $\K^{\bot}/\K$.
 \end{enumerate}
 \end{lemma}
 \begin{proof}

Since $\A$ is cyclic, by Lemma \ref{genlemma}, the extension  $\widetilde{M}_a\supset M\oplus \Z a$ gives rise to an anti-isometry
\begin{equation*}
    \psi'\colon \langle\kappa\rangle\rightarrow \langle n\alpha\rangle\subset\A
\end{equation*}
where $n$ divides $a^2=\operatorname{order}(\alpha)$; hence, $\operatorname{order}(\kappa)=\operatorname{order}(n\alpha)=a^2/n$ and  $\kappa^2=(n\alpha)^2=-n^2/a^2$.

Let $m=a^2/n$. Then $m\alpha\in\K^{\bot}$ and $t_a(m\alpha)=-m\alpha$. Hence, the condition $t_a(m\alpha)=m\alpha \bmod \K$ implies $2m\alpha=0$, \emph{i.e.}, $n=\operatorname{order}(\alpha)/m$ is $1$ or $2$.

In view of Lemma \ref{genlemma} again, statements $(2)$ and $(3)$ are a paraphrase of the condition that $\varphi\oplus t_a$ should extend $\id$ on $\disc \widetilde{M}_a$.
\end{proof}

Lemma \ref{main.lemma}, gives pairs $(\varphi,a)$ such that the involution $\varphi\oplus t_a$ extends to unimodular primitive extensions of $\widetilde{M}_a$. The only question remaining is whether such an extension $L\supset \widetilde{M}_a\oplus N'$ exists. The answer is given by Nikulin's existence theorem applied to the genus with discriminant $-(\K^{\bot}/\K)$ and signature $(2,19-\sigma_-M)$. We denote this genus depending on $M$ and the pair $(\kappa,n)$ by $\tilde{g}_n(M,\kappa)$ (see, section \ref{The.embedding}).
We have the following corollaries applied to $M=\tilde{S}_h$, for which instead of $O(M)$ we restrict to $\varphi\in O_{\pm h}(\tilde{S}_h)$
\begin{corollary}\label{corlem1}
Let a pair $(n,\kappa)$  and $\varphi\in O_{\pm h}(\tilde{S}_h)$ be as in the conclusion of Lemma \ref{main.lemma} and assume $\tilde g_n(\tilde{S}_h,\kappa)\neq\emptyset$, then $\tilde{S}_h$ extends to an abstract homological type admitting an involutive skew-automorphism.
\end{corollary}
\begin{corollary}\label{corlem2}
Let $\operatorname{rk} \tilde{S}_h=18$. Then, $\tilde{S}_h$ extends to an abstract homological type admitting an involutive skew-automorphism if and only if there exists a pair $(n,\kappa)$ and $\varphi\in O_{\pm h}(\tilde{S}_h)$ satisfying Lemma \ref{main.lemma} and $\tilde g_n(\tilde{S}_h,\kappa)\neq\emptyset$.
\end{corollary}
\begin{warning}
Corollary \ref{corlem1} and Corollary \ref{corlem2} do not say anything about any particular abstract homological type. However, typically those corollaries will be applied in the cases when the abstract homological type extending $\tilde{S}_h$ is unique.
\end{warning}
Now we consider $p$-primary components $\kappa_{[p]}$ of the vector $\kappa\in \disc M$ as in the conclusion of Lemma \ref{main.lemma}. The cyclic group $\langle\kappa\rangle$ generated by $\kappa$ decomposes into orthogonal sum $\langle\kappa\rangle=\bigoplus_p\langle\kappa_{[p]}\rangle$ of its $p$-primary components. Note that if $p$ is odd then the group $\langle\kappa_{[p]}\rangle$ is nondegenerate.
\begin{lemma}\label{kappaperb}
Given $n$ and $\kappa$ as in the conclusion of Lemma \ref{main.lemma}, if $p\neq 2$ or $n=1$ then $(\K^{\bot}_{[p]}/\K_{[p]})\cong\kappa_{[p]}^{\bot} $ where $\kappa_{[p]}^{\bot}$ is the orthogonal complement in $\disc_p M$.
\end{lemma}
\begin{proof}
Since $\K\cap\disc S=0$, we have $\K_{[p]}\cap\kappa_{[p]}^{\bot}=0$. Hence the projection map from $\kappa_{[p]}^{\bot}$ to $\K^{\bot}_{[p]}/\K_{[p]}$ is injective. Note that $|\K^{\bot}_{[p]}||\K_{[p]}|=|\disc_p M\oplus \A_{[p]}|$ which implies $|\K^{\bot}_{[p]}/\K_{[p]}|=|\disc_p M||\A_{[p]}|/|\K_{[p]}|^2$. Since $|\A_{[p]}|=|\K_{[p]}|$ for $p\neq 2$ or $n=1$, we obtain $|\K^{\bot}_{[p]}/\K_{[p]}|=|\disc_p M|/|\K_{[p]}|$ . We also have $|\kappa_{[p]}||\kappa_{[p]}^{\bot}|=|\disc_p M|$. Since $|\kappa_{[p]}|=|\K_{[p]}|$, we get $|\kappa_{[p]}^{\bot}|=|\disc_p M|/|\K_{[p]}|$. It follows that $(\K^{\bot}_{[p]}/\K_{[p]})\cong\kappa_{[p]}^{\bot} $.
\end{proof}
\begin{corollary}
Let a pair $(n,\kappa)$  and $\varphi\in O_{\pm h}(\tilde{S}_h)$ be as in the conclusion of Lemma \ref{main.lemma} and assume $\tilde g_n(\tilde{S}_h,\kappa)\neq\emptyset$, and $p\neq 2$ or $n=1$. Then the involution $\varphi\oplus t_a$ extends to $L$ if and only if
\begin{enumerate}
\item $\varphi(\kappa)=-\kappa$ in $\disc \tilde{S}_h$,
\item $\varphi$ induces $\id$ on $\kappa_{[p]}^{\bot}$.
\end{enumerate}
\end{corollary}
\subsection{The embedding $M\oplus\Z a\hookrightarrow L$}\label{The.embedding}
In the previous section we discussed the conditions for the involution $\varphi\oplus t_a$ to extend to a unimodular primitive extensions $L\supset\widetilde{M}_a$, provided that the latter exists, \emph{i.e.} $\tilde{g}_n(M,\kappa)\neq\emptyset$. Now, we analyze this existence.

The isotropic subgroup $\K\subset\disc M\oplus\mathcal{A}$ given as in the conclusion of Lemma \ref{main.lemma} decomposes into orthogonal sum of  its $p$-primary components $\K_{[p]}\subset \disc_p M\oplus\mathcal{A}_{[p]}$ generated by $\kappa_{[p]}+n\alpha_{[p]}$ where $\alpha_{[p]}$ is a generator of $\mathcal{A}_{[p]}$ and $n=1$ or $2$.
\begin{lemma}\label{orth.summand}
Given a primitive extension $L\supset M$ and a pair $(n,\kappa)$ as in the conclusion of Lemma \ref{main.lemma}, if $p$ is odd or $n=1$, then the group generated by $\kappa_{[p]}$ is an orthogonal direct summand, i.e., $\disc M\cong\bar{\M}\oplus\langle\kappa_{[p]}\rangle$.
\end{lemma}
\begin{proof}
Let $p$ be an odd prime  or $n=1$, then we have  $\operatorname{order}(\kappa_{[p]})=\operatorname{order}(\kappa_{[p]}^2)$. Then the form generated by $\kappa_{[p]}$ is nondegenerate and hence an orthogonal direct summand in any form.
\end{proof}
\begin{corollary}\label{coronNiku}
 Given a primitive extension $L\supset M$ and a pair $(n,\kappa)$ as in the conclusion of Lemma \ref{main.lemma}, the hypotheses of Theorem \ref{th.N.existence} for the extension $L\supset M_a$ hold automatically for
  \begin{itemize}
    \item all odd primes $p\mathrel |a^2$;
    \item $p=2$ provided that $n=1$ and parity does not change, \emph{i.e.}, $\disc M$ and $\kappa^{\bot} \subset \disc M$ are either both even or both odd.
  \end{itemize}
\end{corollary}
\subsection{The $2$-primary part}
By Lemma~\ref{orth.summand}, the only nontrivial case (\emph{i.e.}, $\langle\kappa_{[p]}\rangle$ is not an orthogonal summand) is when $p=2$ and $n=2$.
From now on, to avoid more than one subscript, we often abbreviate $\M:=\disc_2 M$, $\K:=\K_{[2]}$, $\kappa:=\kappa_{[2]}$ and $\alpha:=\alpha_{[2]}$ for the corresponding $2$-primary parts. In this notation, we have $\operatorname{ord}(\alpha)=2^{m+1}$, $\alpha^2=\frac{\delta}{2^{m+1}}$ for some positive integer $m$, and, hence,
\begin{align}\label{kappasquare}
    \text{$\operatorname{ord}(\kappa)=2^m$, $\kappa^2=\frac{\xi}{2^{m-1}}$, $\xi$ is odd}.
\end{align}
By applying Lemma \ref{sequences} to the $2$-subgroup $\mathcal{C}=\langle\kappa\rangle$ of $\M=\disc_2 M$, we arrive at the decomposition $\M\cong\bar{\M}\oplus \bigoplus \mathcal{N}_s$.
\begin{observation}\label{observation1}
Recall that $\M=\bigoplus\M_i$, where $\M_i$ is the homogenous group of exponent $2^i$. For any $\sigma_i'\in 2^r\M_i$, we have
\begin{equation}\label{observation1.1}
    (\sigma'_i)^2=\frac{\mu'}{2^{i-2r}},\, \mu'\in\Z.
\end{equation}
Equivalently, for any $\sigma\in 2^{r}\M$ such that $\ord(\sigma)\leq 2^d$, we have
\begin{equation}\label{observation1.2}
    (\sigma)^2=\frac{\mu}{2^{d-r}},\,\mu\in\Z.
\end{equation}
\end{observation}
\begin{observation}\label{observation2}
In view of \eqref{kappasquare}, only the following three homogenous components contribute to $\kappa^2$:
$\kappa'_{m-1}$, $\kappa'_{m}$ and $\kappa'_{m+1}=2\bar{u}_{m+1}$ where $\kappa'_{m-1}\in\M_{m-1}$ and $\bar{u}_{m+1}\in\M_{m+1}$ are orthogonal direct summands whereas $\kappa'_{m}$ is \emph{not}: $\ord(\kappa'_m)=2^m$ and $(\kappa')^2=\xi'/2^{m-1}$.
\end{observation}
For $r_1$ as in \eqref{r1}, by Observation \ref{observation1}, we have $\kappa^2=\bar{\kappa}^2_1=\bar{\xi}/2^{m-r_1}$, $\bar{\xi}\in \mathbb{Z}$, and hence,  either $r_1=0$ or $r_1=1$ by \eqref{kappasquare}.

\smallskip
\noindent\textbf{{The case $r_1=0$}}: Following the construction in Lemma \ref{sequences}, let
\begin{equation*}
    n=\operatorname{max}\{i:\operatorname{ord}(\kappa'_i)=2^{i} \}\leq m\mbox{ and } m_1=n=\operatorname{log}_2\operatorname{ord}(\kappa_n').
\end{equation*}
We have $2^0u_1=\kappa_1=\Sigma_{\ord(\kappa'_i)\leq2^n}\kappa'_i$. We consider three cases: $m_1=m$ or $m_1= m-1$ or $m_1\leq m-2$.

\smallskip
$\bullet$\textit{ The case $n:=m_1=m$}:
Then we have $\kappa=\bar{\kappa}_1=\kappa_1=u_1$ (see \eqref{kappa1}) and the process terminates: $\M=\bar{\M}\oplus\N_1$. By \eqref{kappasquare}, we have $\ell(\N_1)=2$, \emph{i.e.},
\begin{equation}\label{N1.1}
     \M\cong\bar{\M}\oplus\frac{1}{2^{m}}\left[ \begin{array}{cc}
                          \mu_1 & 1 \\
                           1 &  \nu_1 \\
                           \end{array}
                    \right],\,\mu_1\in2\mathbb{Z};\quad\kappa_{[2]}=u_1.
\end{equation}

$\bullet$ \textit{The case $n:=m_1=m-1$}: Then we have $\kappa_1=u_1$, $\operatorname{ord}(\kappa_1)=2^{m-1}$ and $\kappa_1^2=\mu_1/2^{m-1}$.

Consider $\bar{\kappa}_2=\bar{\kappa}_1-\kappa_1$. By construction, $\bar{\kappa}_2=\bigoplus_i\kappa'_i$ where $\kappa_i'\in\M_{i}$ with $i>m-1$ for all $i$. We have $r_2\geq 1$, and the inequality $m_1=m-1<m_2\leq m$ implies that $m_2=m$. Hence $\kappa_2=\bar{\kappa}_2$ (see \eqref{kappa1}) and the algorithm terminates: $\kappa=\kappa_1\oplus\kappa_2$.
It follows that
 $$\kappa_1^2+\kappa_2^2=\frac{\mu_1}{2^{m-1}}+ \frac{\mu_2}{2^{m-r_2}}$$
and, essentially by Observation \eqref{observation2}, we have the following possibilities:

either $r_2=1$, $\mu_1$ is odd and $\mu_2$ is even; then
 \begin{equation}\label{N1N2.1}
        \M\cong\bar{\M}\oplus\frac{1}{2^{m-1}}\left[
       \begin{array}{c}
         \mu_1\\
       \end{array}
     \right]\oplus\frac{1}{2^{m+1}}\left[ \begin{array}{cc}
                          \mu_2 & 1 \\
                           1 &  \nu_2 \\
                           \end{array}
                    \right];\quad\kappa_{[2]}=u_1\oplus2u_2,
\end{equation}

 or $r_2=1$, $\mu_1$ is even and $\mu_2$ is odd; then
  \begin{equation}\label{N1N2.2}
    \M\cong\bar{\M}\oplus\frac{1}{2^{m-1}}\left[ \begin{array}{cc}
                          \mu_1 & 1 \\
                           1 &  \nu_1 \\
                           \end{array}
                    \right]\oplus\frac{1}{2^{m+1}}\left[
       \begin{array}{c}
         \mu_2\\
       \end{array}
     \right];\quad\kappa_{[2]}=u_1\oplus2u_2,
  \end{equation}

or $r_2>1$, then $\mu_1$ is odd and $\mu_2$ can be odd or even; then
\begin{equation}\label{N1N2.3}
    \M\cong\bar{\M}\oplus\frac{1}{2^{m-1}}\left[
       \begin{array}{c}
         \mu_1\\
       \end{array}
     \right]\oplus\N_2;\quad\kappa_{[2]}=u_1\oplus2^{r_2}u_2,
\end{equation}
where $\N_2$ is either
\begin{equation}\label{N1N2.4}
    \frac{1}{2^{m+r_2}}\left[
       \begin{array}{c}
         \mu_2\\
       \end{array}
     \right]\quad\mbox{or}\quad\frac{1}{2^{m+r_2}}\left[ \begin{array}{cc}
                          \mu_2 & 1 \\
                           1 &  \nu_2 \\
                           \end{array}
                    \right].
\end{equation}

$\bullet$\textit{ The case $n:=m_1\leq m-2$}: By Observation \ref{observation2}, at the next step we have $r_2=1$ and $n_2\geq m+1$, hence again $m_2=m$, and the algorithm terminates: $\kappa=\kappa_1\oplus\kappa_2$, hence we have
 $$\kappa_1^2+\kappa_2^2=\frac{\mu_1}{2^{m-2}}+ \frac{\mu_2}{2^{m-1}}$$
and by Observation \ref{observation2} again, $\mu_2$ is odd and $\mu_1$ can be odd or even. Then, we obtain
\begin{equation}\label{N1N2.5}
    \M\cong\bar{\M}\oplus\N_1\oplus\frac{1}{2^{m+1}}\left[
       \begin{array}{c}
         \mu_2\\
       \end{array}
     \right]\quad\kappa_{[2]}=u_1\oplus2u_2
\end{equation}
where $\N_1$ is either
\begin{equation}\label{N1N2.6}
    \frac{1}{2^{n}}\left[
       \begin{array}{c}
         \mu_1\\
       \end{array}
     \right], \, n\leq m-2\quad\mbox{or}\quad\frac{1}{2^{n}}\left[ \begin{array}{cc}
                          \mu_1 & 1 \\
                           1 &  \nu_1 \\
                           \end{array}
                    \right],\, n\leq m-2.
\end{equation}
\noindent\textbf{The case $r_1=1$}: Then we have
\begin{equation*}
    n=\operatorname{max}\{i:\operatorname{ord}(\kappa'_i)=2^{i-1} \}\leq m+1.
\end{equation*}
Since $n>m$, by Observation \ref{observation2}, we conclude that $n=m+1$ and $m_1=m$, \emph{i.e.}, the algorithm terminates at the first step and we arrive at
\begin{equation}\label{N1N2.7}
     \M\cong\bar{\M}\oplus\frac{1}{2^{m+1}}\left[ \begin{array}{c}
                          \mu_1
                           \end{array}
                    \right];\quad\kappa_{[2]}=2u_1
\end{equation}

\subsection{The group $\K^{\bot}/\K$}\label{description.KperbmodK}
For the decompositions $\M\cong\bar{\M}\oplus \bigoplus_{s=1}^N \mathcal{N}_s$ given above of length $N\leq2$, the corresponding groups $\K^{\bot}/\K$ which are the orthogonal direct sum of $\bar{\M}$ and a subgroup generated by certain elements $\{w_i\}$ are described below (in terms of the notation given in Lemma \ref{sequences}) on a case by case basis. We also indicate
the ``\emph{ambiguous}" cases where $\M$ is odd and $\K^{\bot}/\K$ is even (when describing the \emph{ambiguous} cases, we assume $\bar{\M}$ is even since otherwise both $\M$ and $\K^{\bot}/\K$ are odd forms ):

The case $r_1=0$ :

\begin{itemize}
\item In case \eqref{N1.1}, $\K=\langle u_1+2\alpha\rangle$ and $\K^{\bot}/\K\cong\bar{\M}\oplus\mathbb{Z}/2^{m+1}$, generated by $w_1=\alpha-\delta v_1$. This case is \emph{ambiguous} if $m=1$ and $\nu_1$ is odd.

\item In case \eqref{N1N2.1}, $\K=\langle u_1+2u_2+2\alpha\rangle$ and $\K^{\bot}/\K\cong\bar{\M}\oplus\mathbb{Z}/2^{m+1}\oplus\mathbb{Z}/2^{m+1}$, generated by $w_1=\delta v_2-\alpha$ and $w_2=\delta u_2-\mu_2\alpha$.  This case is \emph{ambiguous} if $m=2$.

\item In case \eqref{N1N2.2}, $\K=\langle u_1+2u_2+2\alpha\rangle$ and $\K^{\bot}/\K\cong\bar{\M}\oplus\mathbb{Z}/2^{m}\oplus\mathbb{Z}/2^{m}$, generated by $w_1=v_1-\frac{2}{\mu_2}u_2$ and $w_2=\frac{\mu_1}{2}v_1+u_2+\alpha$. This case is \emph{ambiguous} if $m=1$ and $\nu_1$ is odd.

 \item In case \eqref{N1N2.3} with the former case of \eqref{N1N2.4},  $\K=\langle u_1+2^{r_2}u_2+2\alpha\rangle$ and $\K^{\bot}/\K\cong\bar{\M}\oplus\mathbb{Z}/2^{m+r_2}$, generated by $w_1=\delta u_2-\mu_2\alpha$.  This case is \emph{ambiguous} if $m=2$.

  \item In case \eqref{N1N2.3} with the latter case of \eqref{N1N2.4}, $\K=\langle u_1+2^{r_2}u_2+2\alpha\rangle$ and $\K^{\bot}/\K\cong\bar{\M}\oplus\mathbb{Z}/2^{m+r_2}\oplus\mathbb{Z}/2^{m+r_2}$, generated by $w_1=\delta v_2-\alpha$ and $w_2=\delta u_2-\mu_2\alpha$. This case is \emph{ambiguous} if $m=2$.
  \item In case \eqref{N1N2.5} with the former case of \eqref{N1N2.6}, $\K=\langle u_1+2u_2+2\alpha\rangle$ and $\K^{\bot}/\K\cong\bar{\M}\oplus\mathbb{Z}/2^{n+2}$, generated by $w_1=-\delta u_2+\mu_2\alpha$.  This case is \emph{ambiguous} if $m\geq3$ and $n=1$.
  \item In case \eqref{N1N2.5} with the latter case of \eqref{N1N2.6}, $\K=\langle u_1+2u_2+2\alpha\rangle$ and $\K^{\bot}/\K\cong\bar{\M}\oplus\mathbb{Z}/2^{n+1}\oplus\mathbb{Z}/2^{n+1}$, generated by $w_1=\mu_2v_1-2^{m-n}u_2$ and $w_2=\frac{\mu_1}{2}v_1+u_2+\alpha$.  This case is \emph{ambiguous} if $n=1$ and $\nu_1$ is odd.
\end{itemize}
The case $r_1=1$:
\begin{itemize}
\item   In case \eqref{N1N2.7}, $\K=\langle 2u_1+2\alpha\rangle$ and $\K^{\bot}/\K\cong\bar{\M}\oplus\mathbb{Z}/2\oplus\mathbb{Z}/2$, generated by $w_1=u_1+\alpha$ and $w_2=2^mu_1$.
\end{itemize}

\section{Applications}\label{Applications}
\subsection{Simple Quartics}\label{simple.quartics}
In this section we consider birational projective models $f_h\colon X\rightarrow \mathbb{P}^3$ with $h^2=4$, \emph{i.e.}, spatial model.
The image is a quartic surface in $\mathbb{P}^3$.
For a simple quartic $X$, the minimal resolution of singularities $\tilde{X}$ is a smooth $K3$-surface; hence the intersection lattice is of the form $H_2(\tilde{X})\cong \L$.
\begin{definition}
A quartic $X$ is called \emph{nonspecial} if the abstract homological type $\mathcal{H}(S\oplus \Z h\subset L)$ associated to the projective model $f_h\colon X\rightarrow \mathbb{P}^3$ is primitive, \emph{i.e}, $S_h\subset L$ is a primitive extension.
\end{definition}
For a given set of simple singularities $S$, the corresponding equisingular stratum of quartics is denoted by $\mathcal{M}(S)$. Our primary interest is the family $\mathcal{M}_1(S)\subset \mathcal{M(S)}$ constituted by the nonspecial quartics with the set of singularities $S$.

A complete description of the strata $\mathcal{M}_1(S)$ of nonspecial simple quartics is given by G\"{u}ne\c{s} Akta\c{s} \cite{Cisem1}, where it is proved that
the strata  $\mathcal{M}_1 (S)$ with $S=\mathbf{D}_6\oplus2\mathbf{A}_6$, $\mathbf{D}_5\oplus2\mathbf{A}_6\oplus \mathbf{A}_1$, $2\mathbf{A}_7\oplus 2\mathbf{A}_2$, $3\mathbf{A}_6$ or $2\mathbf{A}_6\oplus2\mathbf{A}_3$ split into pairs of complex conjugate components; all other non-maximizing equisingular strata are connected (i.e., they consist of one real component). The classification of the connected components of the $59$ maximizing strata is also available there.  As one of the main applications of this paper, we give the proof of Theorem~\ref{principal.result}.
\begin{proof}[Proof of Theorem~\ref{principal.result}]
By  Theorem \ref{real.model}, the question reduces to finding an involutive skew-automorphism of the primitive abstract homological type $\mathcal{H}=(S\oplus \Z h\subset L)$ extending the root lattice $S$.
\begin{remark}\label{sym}
Since the homological type is primitive we have $\tilde{S}_h=S_h$, $\disc \tilde{S}_h=\disc S\oplus [ \frac{1}{4}]$ and $O_h(\tilde{S}_h)=O(S)$. Furthermore, since we are interested in the induced action on discriminant, the group $O_{\pm h}(\tilde{S}_h)$ can be replaced with $\operatorname{Sym}(\Gamma)\times\{\pm \id_{h}\}$, where $\Gamma$ is the Dynking diagram of the root lattice $S$ (\emph{cf}. section \ref{root.lattices})
\end{remark}
If $\operatorname{rk} S=19$, the statement of the theorem is given by Proposition \ref{rankT=2}. Hence, throughout the rest of the proof we assume $\operatorname{rk} S\leq18$.

By computer aided computations, it is easily confirmed that most of the abstract homological types $\mathcal{H}=(S\oplus\Z h\subset \L )$ (except $100$ of them) with $\operatorname{rk} S\leq18$ are symmetric $\operatorname{c}$-invariant perturbations of the $37$ maximizing primitive homological types (see, \cite{Cisem1}) realized by a real quartic where $c$ is a real structure on the corresponding real surface. Then, due to Proposition \ref{real.pert}, these abstract homological types are also realized by a real nonspecial quartic.

The space $\mathcal{M}_1 (S)$ with $S=\mathbf{D}_6\oplus2\mathbf{A}_6$, $\mathbf{D}_5\oplus2\mathbf{A}_6\oplus \mathbf{A}_1$, $2\mathbf{A}_7\oplus 2\mathbf{A}_2$, $3\mathbf{A}_6$, $2\mathbf{A}_6\oplus2\mathbf{A}_3$ consists of two complex conjugate components (see, \cite{Cisem1}). Therefore the strata $\mathcal{M}_1(S)$ do not contain a real surface.

For each of the remaining $95$ sets of singularities $S$, we used GAP~\cite{GAP} to find a positive integer $a^2\mathrel| 2\operatorname{exp}(\disc \tilde{S}_h)$, integer $n=1$ or $2$, class $\kappa\in \disc \tilde{S}_h$ and involution $\varphi\in O_{\pm h}(\tilde{S}_h)$ satisfying the conditions in Lemma \ref{main.lemma} and such that $\tilde{g}_n(\tilde{S}_h, \kappa)\neq\emptyset$. If found, Corollary \ref{corlem1} and Theorem \ref{real.model} imply that the stratum $\mathcal{M}_1(S)$ contains a real surface. Since, on the other hand, $\M_1(S)$ is connected (see Corollary 4.2.4 in \cite{Cisem1}), this implies the statement.

The above algorithm fails for the set of singularities $S_1=\mathbf{A}_7\oplus\mathbf{A}_6\oplus\mathbf{A}_3\oplus\mathbf{A}_2$ and $S_2=\mathbf{D}_7\oplus\mathbf{A}_6\oplus\mathbf{A}_3\oplus\mathbf{A}_2$ (exceptional cases  listed in the statement), \emph{i.e.}, computer aided calculations confirm that there does not exist a pair $(n,\kappa)$ and $\varphi\in O_{\pm h}(S_h)$ satisfying Lemma \ref{main.lemma} and $\tilde g_n(S_h,\kappa)\neq\emptyset$. Since $\operatorname{rk} S_1=\operatorname{rk} S_2 =18$, by Corollary \ref{corlem2}, the statement follows. $\qed$

Although the implemented calculations by GAP~\cite{GAP} completes the proof for the exceptional cases $S_1$ and $S_2$,  in the following two subsections, we provide explicit details, just to illustrate how many things may go wrong in  constructing  a skew-automorphism on the abstract homological types extending $S_1$ and $S_2$.

Before continuing, we make the following observation to which we will refer in the proof for the two exceptional cases.
\begin{observation}\label{observation123}
Assume that $\operatorname{rk} S=18$ and the group $\disc_2(\tilde{S}_h\oplus \mathbb{Z}a)$ is the orthogonal sum of cyclic groups with generators $\alpha_i$, where $\alpha_0=\alpha_{[2]}$ is the generator of $\disc_2\mathbb{Z}a$, and $\operatorname{order}(\alpha_i)=4$ or $8$ for all $i$. Assume further that the action of $O_{\pm h}(\tilde{S}_h)\times\{\pm \id_a\}$ on $\disc (\tilde{S}_h\oplus\mathbb{Z}a)$ is generated by involutions $\alpha_i\mapsto\pm\alpha_i$. Let the $2$-primary part of the kernel $\K_{[2]}$ be generated by a single element $\vartheta=\kappa_{[2]}+n\alpha_{[2]}$ of the form $k\alpha_i+l\alpha_j+\vartheta'$, $k,l\in\Z$ where each $k\alpha_i$, $l\alpha_j$ and $\vartheta'$ has order at least $4$ and $\vartheta'$ is a combination of generators other than $\alpha_i,\alpha_j$. Then, in the three cases considered below, the involution $\varphi\oplus t_a$ reverses $\alpha_i$ and $\alpha_j$ (recall that $\varphi(\kappa)=-\kappa$) and the element $\nu$ described below has order at least $4$ in $\K^{\bot}/\K$ and is reversed by $\varphi\oplus t_a$, contrary to Lemma~\ref{main.lemma}(3). Hence by Corollary~\ref{corlem2}, $\tilde{S}_h$ does not extend to an abstract homological type admitting a skew-automorphism.
\begin{enumerate}
  \item If $\vartheta=(\pm \alpha_i\pm\alpha_j)+\vartheta'$ with $\operatorname{order}(\alpha_j)=4$, then $\nu$ is one of $\alpha_i\pm\alpha_j$
  \item If $\vartheta=(\pm 2\alpha_i\pm\alpha_j)+\vartheta'$ with $\operatorname{order}(\alpha_j)=4$, then $\nu$ is one of $\alpha_i\pm\alpha_j$
  \item If $\vartheta=(\pm \alpha_i\pm\alpha_j)+\vartheta'$ with $\operatorname{order}(\alpha_j)=4$, then $\nu$ is one of $2\alpha_i\pm\alpha_j$
\end{enumerate}
\end{observation}
\subsection{The set of singularities $S=\mathbf{D}_7\oplus\mathbf{A}_6\oplus\mathbf{A}_3\oplus\mathbf{A}_2$}\label{exceptional1}\label{S1}
One has
$$\disc \tilde{S}_h = \disc S_h\cong\textstyle[\frac{1}{4}]\oplus[-\frac{6}{7}]\oplus[-\frac{3}{4}]\oplus[-\frac{2}{3}]\oplus[
\frac{1}{4}].$$
Consider an integer $a^2\mathrel| 2\operatorname{exp} (\disc \tilde{S}_h)= 2^3\cdot3\cdot7$, \emph{i.e.}, $a^2= 2^N\cdot3^r\cdot7^s$, where $N\in\{1,2,3\}$, $r\in\{0,1\}$ and $s\in\{0,1\}$, see \eqref{adivides2exp}. With $a^2$ fixed, consider a pair $(n,\kappa)$, where $n=1$ or $2$ and $\kappa\in \disc_2\tilde{ S}_h$, as in Lemma \ref{main.lemma}. Fix the generators
$$\text{$\alpha_1$ for $\disc \mathbf{D}_7\cong\textstyle[\frac{1}{4}]$,\quad $\alpha_2$ for $\disc \mathbf{A}_3\cong\textstyle[-\frac{3}{4}]$,\quad$\alpha_3$ for $\disc \mathbb{Z}h\cong\textstyle[\frac{1}{4}]$},$$
for the $2$-primary part. It is immediate (\emph{cf}. Remark~\ref{sym}) that the action of $O_{\pm h}(\tilde{S}_h)$ on $\disc \tilde{S}_h$ is generated by the involutions $\alpha_i\mapsto\pm\alpha_i$ with $i=1,2,3$, as in Observation \ref{observation123}.
\subsubsection{The case $N=1$}\label{subN1}
Then $n=2$ (as $\disc \tilde{S}_h$ does not contain any cyclic direct summand of order $2$). Hence, we have $\kappa_{[2]}=0$ and $\ell_2(\K^{\bot}/\K)= 4$ , which implies $\tilde{g}_n(\tilde{S}_h,\kappa)=\emptyset$ by Theorem \ref{th.N.existence}. 
\subsubsection{The case $N=2$ and $n=2$}\label{subN2n2}
Then we have $\ell_2(\K^{\bot}/\K)= 4$ by \eqref{N1N2.7} and the respective item in section \ref{description.KperbmodK}, which implies $\tilde{g}_n(\tilde{S}_h,\kappa)=\emptyset$ by Theorem \ref{th.N.existence}.
\subsubsection{The case $N=2$ and $n=1$}\label{N2n1}
For an odd prime $p$, the group $\langle\kappa_{[p]}\rangle$ is an orthogonal summand in the cyclic group $\disc_p \tilde{S_h}$ (by Lemma~\ref{orth.summand}), hence, in this particular case, we have either $\kappa_{[p]}=0$, which implies $ p\nmid a^2$, or $\langle\kappa_{[p]}\rangle\cong\disc_p\tilde{S}_h$. Since $\disc_p\tilde{S}_h$ is a cyclic group, this implies an extra condition on $a^2$: one must have $a^2/3= -2 \bmod (\mathbb{Z}_2^{\times})^2$ and $a^2/7= -6 \bmod (\mathbb{Z}_7^{\times})^2$. By checking these conditions,, we  rule out the cases $a^2=4\cdot3,4\cdot7,4\cdot3\cdot7$ and obtain $a^2=4$. Listing  $\kappa_{[2]}\in \disc_2 \tilde{S}_h$ with
$$
   \text{$\operatorname{order }(\kappa_{[2]})=4$ and $\kappa^2_{[2]}=-\frac{1}{4}$},
$$
we arrive at $\vartheta=(\alpha_0\pm\alpha_1)\pm\alpha_2\pm\alpha_3$ as in Observation~\ref{observation123}(1), ruling this case out.
\subsubsection{The case $N=3$}
Then $n=2$ (as $\disc_2 \tilde{S}_h$ does not contain any cyclic direct summand of order $8$). As in section~\ref{N2n1}, we rule out the case $a^2=8\cdot7$ and obtain $a^2=8\delta$,  where $\delta=1,3$ or $21$. Then, listing all the vectors $\kappa_{[2]}\in \disc_2\tilde{S}_h$ satisfying
$$
   \text{$\operatorname{order }(\kappa_{[2]})=4$ and $\kappa_{[2]}^2=-\frac{\delta}{2}$},
$$
we obtain $\vartheta=(2\alpha_0\pm\alpha_i)\pm\alpha_j$ or $(2\alpha_0\pm\alpha_i)\pm\alpha_j+2\alpha_k$ as in Observation~\ref{observation123}(2), eliminating this case.
\subsection{The set of singularities $S=\mathbf{A}_7\oplus\mathbf{A}_6\oplus\mathbf{A}_3\oplus\mathbf{A}_2$}\label{S2}
One has
$$\disc \tilde{S}_h=\disc S_h\cong\textstyle[-\frac{7}{8}]\oplus[-\frac{6}{7}]\oplus[-\frac{3}{4}]\oplus[-\frac{2}{3}]\oplus[\frac{1}{4}].$$
Consider an integer $a^2\mathrel| 2\operatorname{exp} (\disc \tilde{S}_h)= 2^4\cdot3\cdot7$, \emph{i.e.}, $a^2= 2^N\cdot3^r\cdot7^s$, where $N\in\{1,2,3,4\}$, $r\in\{0,1\}$ and $s\in\{0,1\}$. As above, fix $a^2$ and consider a pair $(n,\kappa)$ (where $n=1,2$) and $\kappa_{[2]}\in \disc_2\tilde {S}_h$, as in Lemma \ref{main.lemma}.
Fix the generators
$$\text{$\alpha_1$ for $\disc \mathbf{A}_7\cong\textstyle[-\frac{7}{8}]$,\quad $\alpha_2$ for $\disc \mathbf{A}_3\cong\textstyle[-\frac{3}{4}]$,\quad$\alpha_3$ for $\disc \mathbb{Z}h\cong\textstyle[\frac{1}{4}]$},$$
for $\disc_2 \tilde{S}_h$. By Remark~\ref{sym}, the group $O_{\pm h}(\tilde{S}_h)$ acts on $\disc \tilde{S}_h$ via the involutions $\alpha_i\mapsto\pm\alpha_i$ with $i=1,2,3$, as in Observation~\ref{observation123}.

\subsubsection{The case $N=1$ is ruled out as in section~\ref{subN1}}

\subsubsection{The case $N=2$, $n=2$ is ruled out as in section~\ref{subN2n2}}
\subsubsection{The case $N=2$ and $n=1$}
We obtain $a^2=4$ as in section~\ref{N2n1}. (Note that $3$- and $7$-primary parts of $\disc \tilde{S}_h$ are the same as in section~\ref{exceptional1}). Then, listing all vectors $\kappa_{[2]}\in \disc_2\tilde{S}_h$ with
$$\text{$\operatorname{order }(\kappa_{[2]})=4$ and $\kappa^2_{[2]}=-\frac{1}{4}$},$$
we arrive at $\vartheta=(\alpha_0\pm\alpha_3)\pm2\alpha_1\pm2\alpha_2$ or $(\alpha_0\pm\alpha_2)\pm2\alpha_1$ as in Observation~\ref{observation123}(1).

\subsubsection{The case $N=3$ and $n=1$}\label{N3n1(2)}
As in section~\ref{N2n1}, we have $a^2=8\delta$,  where $\delta=1,3$ or $21$. We search for $\kappa_{[2]}\in \disc_2\tilde{ S}_h$ with
$$
   \text{$\operatorname{order }(\kappa_
   {[2]})=8$ and $\kappa_{[2]}^2=-\frac{\delta}{8}$}.
$$
\par If $\delta=1$, then there is no such element $\kappa_{[2]}$.

If $\delta=3$, then listing all such vectors $\kappa_{[2]}$, we obtain $\vartheta=(\alpha_0\pm\alpha_2)\pm3\alpha_1\pm \alpha_3$ as in Observation~\ref{observation123}(3).

If $\delta=21$, then we have $\vartheta=(\alpha_0\pm\alpha_2)\pm\alpha_1+2\alpha_3$, $(\alpha_0\pm\alpha_3)\pm3\alpha_1+2\alpha_2$, $(\alpha_0\pm\alpha_2)\pm3\alpha_1$ or $(\alpha_0\pm\alpha_3)\pm\alpha_1$ as in Observation~\ref{observation123}(3).
\subsubsection{The case $N=3$ and $n=2$}
By section~\ref{N3n1(2)}, we have $a^2=8\delta$,  where $\delta=1,3$ or $21$. Then, listing $\kappa_{[2]}\in \disc_2\tilde{ S}_h$ such that
$$
   \text{$\operatorname{order }(\kappa_{[2]})=4$ and $\kappa_{[2]}^2=-\frac{\delta}{2}$},
$$
we obtain $\vartheta=(2\alpha_0\pm \alpha_2)+4\alpha_1\pm\alpha_3$, $(2\alpha_0\pm\alpha_2)\pm\alpha_3$ or $2\alpha_0\pm2\alpha_1+2\alpha_3$, $2\alpha_0\pm2\alpha_1+2\alpha_2$. The former two cases are covered by Observation~\ref{observation123}(2). The latter two cases are ruled out as in section \ref{subN2n2}.
\subsubsection{The case $N=4$}\label{sub4}
Then $n=2$, since $\disc_2 \tilde{S}_h$ does not contain any cyclic direct summand of order $16$. On the other hand any order $8$ element in $\disc_2\tilde{S}_h$ is an orthogonal direct summand which contradicts to $n=2$.
 \end{proof}
\subsection{Simple Sextics}
In this section we consider hyperelliptic projective models $f_h\colon X\rightarrow \mathbb{P}^2$ with $h^2=2$, \emph{i.e.}, planar models. Recall that $f_h$ is a degree $2$ map ramified at a sextic curve $C\subset \mathbb{P}^2$.

A complete description of the strata $\mathcal{M}_1(S)$ of nonspecial simple sextics is given in~\cite{Alex2}. Degtyarev and Akyol also showed that the space $\mathcal{M}_1(\mathbf{A}_7\oplus\mathbf{A}_6\oplus\mathbf{A}_5)$ consists of a single component, which is hence real, but it contains no real curves (see Proposition 2.6 in \cite{Alex2}). This result is similar to our main result Theorem~\ref{principal.result} for simple quartics and our approach gives a simpler and more transparent proof which is outlined below.
\begin{proposition}[Proposition 2.6 in~\cite{Alex2}]
The stratum $\mathcal{M}_1(\mathbf{A}_7\oplus\mathbf{A}_6\oplus\mathbf{A}_5)$ contains no real curves.
\end{proposition}
\begin{proof}
One has
$$\disc \tilde{S}_h=\disc S_h\cong\textstyle[-\frac{7}{8}]\oplus[-\frac{6}{7}]\oplus[\frac{2}{3}]\oplus[\frac{1}{2}]\oplus[\frac{1}{2}].$$
(Note that $\operatorname{rk}S_h=19$, hence by Proposition~\ref{rankT=3} it is enough to show there is no reflection). Consider an integer $a^2\mathrel| 2\operatorname{exp} (\disc \tilde{S}_h)= 2^4\cdot3\cdot7$, \emph{i.e.}, $a^2= 2^N\cdot3^r\cdot7^s$, where $N\in\{1,2,3,4\}$, $r\in\{0,1\}$ and $s\in\{0,1\}$. Similar to sections \ref{S1} and \ref{S2}, for a fixed $a^2$, we consider a pair $(n,\kappa)$ (where $n=1,2$) and $\kappa_{[2]}\in \disc_2\tilde {S}_h$, as in Lemma \ref{main.lemma}.
We fix the generators
$$\text{$\alpha_1$ for $\disc \mathbf{A}_7\cong\textstyle[-\frac{7}{8}]$,\quad $\alpha_2$ for $\disc_2 \mathbf{A}_5\cong\textstyle[\frac{1}{2}]$,\quad$\alpha_3$ for $\disc \mathbb{Z}h\cong\textstyle[\frac{1}{2}]$},$$
for $\disc_2 \tilde{S}_h$. By Remark~\ref{sym}, the group $O_{\pm h}(\tilde{S}_h)$ acts on $\disc \tilde{S}_h$ via the involutions $\alpha_i\mapsto\pm\alpha_i$ with $i=1,2,3$, as in Observation~\ref{observation123}.

\subsubsection{The case $N=1$, $n=1$ }\label{sub11}
As in section~\ref{N2n1}, we rule out the cases $a^2=2\cdot3,2\cdot7,2\cdot3\cdot7$ and obtain $a^2=2$. However, there is no element $\kappa_{[2]}\in \disc_2\tilde{S}_h$ such that
$$
   \text{$\operatorname{order }(\kappa_{[2]})=2$ and $\kappa_{[2]}^2=-\frac{1}{2}$},
$$
eliminating this case.
\subsubsection{The case $N=1$, $n=2$ is ruled out as in section~\ref{subN1} }
\subsubsection{The case $N=2$}
Then $n=2$ (as the group $\disc_2\tilde{S}_h$ does not contain any cyclic direct summand of order $4$).
As in section~\ref{N2n1}, we have $a^2=4\delta$,  where $\delta=1,3$ or $3\cdot7$. We search for $\kappa_{[2]}\in \disc_2\tilde{ S}_h$ with
$$
   \text{$\operatorname{order }(\kappa_{[2]})=2$ and $\kappa_{[2]}^2=1$}.
$$
According to \ref{N1.1} and the respective item in section \ref{description.KperbmodK}, $\K^{\bot}/\K$ has a summand $\mathbb{Z}/4$ generated by $\alpha_0+$(element of order $2$) which is reversed by $\varphi\oplus t_a$, contrary to Lemma~\ref{main.lemma}(3), and hence to Corollary~\ref{corlem2}.

\subsubsection{The case $N=3$ and $n=1$}\label{N3n1(3)}
As in section~\ref{N2n1}, we have $a^2=8$ (\emph{cf}. section~\ref{sub11}). Then, there is no $\kappa_{[2]}\in \disc \tilde{S}_h$ with
$\operatorname{order }(\kappa_{[2]})=8$ and $\kappa_{[2]}^2=-\frac{1}{8}$.
\subsubsection{The case $N=3$ and $n=2$}
By section~\ref{N3n1(3)}, we have $a^2=8$. Then, listing all elements $\kappa_{[2]}\in \disc \tilde{S}_h$ with
$$
   \text{$\operatorname{order }(\kappa_{[2]})=4$ and $\kappa_{[2]}^2=-\frac{1}{2}$},
$$
we obtain $\kappa_{[2]}+2\alpha_0=2\alpha_0\pm 2\alpha_1+\alpha_2+\alpha_3$. Arguing as in Observation~\ref{observation123}, we conclude that an involution $\varphi$ with $\varphi(\kappa)=-\kappa$ must send each generator $\alpha_i$, $i=1,2,3$ to $-\alpha_i$, thus inducing $-\id$ on $\K^{\bot}/\K$. On the other hand, one of the vectors $3\alpha_0\pm\alpha_1+\alpha_3$ is in $\K^{\bot}$ and has order $8$ in $\K^{\bot}/\K$. This contradicts to Lemma~\ref{main.lemma}(3), and hence to Corollary~\ref{corlem2}.
\subsubsection{The case $N=4$ is ruled out as in section~\ref{sub4}}
 \end{proof}

\bibliographystyle{amsplain}
\bibliography{mybibliography}

\end{document}